\documentclass[11pt]{amsart}

\usepackage{etex}

\usepackage{amsmath,amsfonts,amscd,flafter,epsf,amssymb}
\usepackage{graphics,xypic}
\usepackage{epsfig}
\usepackage{psfrag}
 \usepackage{graphicx,color}
 \usepackage{wrapfig}
 \usepackage{picins}
 \usepackage{pinlabel}
 \usepackage{subfigure}
 \usepackage{tikz}\usetikzlibrary{matrix,arrows,patterns}
\usepackage{booktabs}
\usepackage{stmaryrd}

\newcommand{\myline}{\par
  \kern -15pt }
  
\newcommand{\G}{\mathcal{G}}

\def\co{\colon\thinspace}

\newcommand{\filta}{\widehat{F}}
\newcommand{\cone}[1]{\operatorname{cone}\left(#1\right)}

\newcommand{\bu}{\bullet}

\newcommand{\bZ}{\mathbb{Z}}
\newcommand{\bF}{\mathbb{F}}

\newcommand{\Khred}{\widetilde{Kh}{}}
\newcommand{\CKhred}{\widetilde{CKh}{}}

\newcommand{\rk}{\operatorname{dim}}

\newcommand{\half}{\frac{1}{2}}

\newcommand{\zero}
	{\raisebox{-2pt}
	{\includegraphics[scale=0.085]{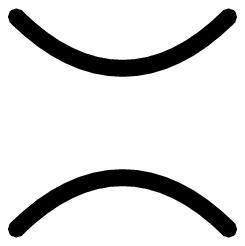}}}

\newcommand{\otherleftcross}
	{\raisebox{-2pt}
	{\rotatebox{90}
	{\includegraphics[scale=0.085]{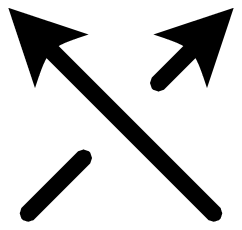}}}}
\newcommand{\otherrightcross}
	{\raisebox{-2pt}
	{\rotatebox{90}
	{\includegraphics[scale=0.085]{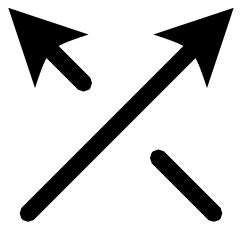}}}}
	
	\newcommand{\otheroriented}
	{\raisebox{-1.5pt}
	{\rotatebox{90}
	{\includegraphics[scale=0.085]{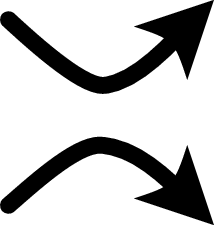}}}}

\newtheorem{thm}{Theorem}
\newtheorem{theorem}{Theorem}[section]
\newtheorem{prop}{Proposition}
\newtheorem{proposition}[theorem]{Proposition}  

\newtheorem{cor}[thm]{Corollary}
\newtheorem{conj}[theorem]{Conjecture}

\newtheorem{defn}[theorem]{Definition}

\newtheorem{question}[theorem]{Question}

\newcommand{\Alex}{\Delta_K(t)}
\newcommand{\Q}{\mathbb{Q}}

\newcommand{\Z}{\mathbb{Z}}

\newcommand{\cm}{\cdot}

\newcommand\SpinC{\mathrm{Spin}^c}
\newcommand{\F}{\mathbb F}

\newcommand\Filt{\mathcal F}

\newcommand\x{\mathbf x}

\newcommand\y{\mathbf y}

\newcommand\ModSphere{\ModFlow\left({\mathbb S}\longrightarrow 
\Sym^{g-1}(\Sigma_{1})\times \Sym^2(\Sigma_{2})\right)}
\newcommand\ModSpheres\ModSphere
\newcommand\CF{CF}

\newcommand\CFp{\CFb}
\newcommand\CFm{\CF^-}

\newcommand\HFp{\HFb}

\newcommand\CFinf{CF^\infty}
\newcommand\HFinf{HF^\infty}
\newcommand\CFb{CF^+}
\newcommand\HFa{\widehat{HF}}
\newcommand\HFb{HF^+}

\newcommand\UnparModSp{\widehat \ModSp}
\newcommand\UnparModFlow\UnparModSp
\newcommand\Mod\ModSp

\newcommand{\spinc}{\mathfrak s}

\newcommand{\spinct}{\mathfrak t}

\newcommand\ModMaps{\mathcal M}
\newcommand\ModSp\ModMaps

\newcommand\CFK{CFK}
\newcommand\HFK{HFK}

\newcommand\CFKp{\CFK^+}

\newcommand\CFKm{\CFK^-}
\newcommand\CFKinf{\CFK^{\infty}}
\newcommand\HFKp{\HFK^+}

\newcommand\HFKa{\widehat\HFK}
\newcommand\HFKm{\HFK^-}
\newcommand\HFKinf{\HFK^{\infty}}

\newcommand\ons{Ozsv{\'a}th and Szab{\'o}}
\newcommand\os{{Ozsv{\'a}th-Szab{\'o}}}

\newcommand{\Kcross}{\mathcal{K}}

\newtheorem*{namedtheorem}{\theoremname}
\newcommand{\theoremname}{testing}
\newenvironment{named}[1]{\renewcommand{\theoremname}{#1}
        \begin{namedtheorem}}
        {\end{namedtheorem}}

\title[On the geography and botany of knot Floer homology] 
{On the geography and botany of knot Floer homology}
\date{\today}

\author[Matthew Hedden]{Matthew Hedden}
\thanks{The first author was partially supported by NSF Grants DMS-0706979 DMS-0906258, CAREER DMS-1150872, and an A.P. Sloan Research Fellowship.}
\address {Department of Mathematics, Michigan State University, MI, USA}
\email {mhedden@math.msu.edu}

\author[Liam Watson]{Liam Watson}
\thanks{The second author was partially supported by a Marie Curie Career Integration Grant (HFFUNDGRP).}
\address{University of Glasgow, School of Mathematics and Statistics, Glasgow, United Kingdom \newline {\em Current address:} Universit\'e de Sherbrooke, D\'epartement de Math\'ematiques, QC, Canada}
\email{liam.watson@usherbrooke.ca}

\begin{document}

\begin{abstract} This paper explores two questions: $(1)$ Which bigraded groups arise as the knot Floer homology of a knot in the three-sphere? $(2)$ Given a knot, how many distinct knots share its Floer homology?  Regarding the first, we show there exist bigraded groups satisfying all previously known constraints of knot Floer homology which do not arise as the invariant of a knot.  This leads to a new constraint for knots admitting lens space surgeries, as well as a proof that the rank of knot Floer homology detects the trefoil knot.  For the second, we show that  any non-trivial band sum of two unknots gives rise to an infinite family of distinct knots with isomorphic knot Floer homology.  We also prove that the fibered knot with identity monodromy is strongly detected by its knot Floer homology, implying that Floer homology solves the word problem for mapping class groups of  surfaces with non-empty boundary.  Finally, we survey some conjectures and questions and, based on the results described above, formulate some new ones.
  \end{abstract}

\maketitle

\section{Introduction}
The Alexander polynomial is a classical invariant of knots \cite{Alexander}.  For a knot $K\subset S^3$, it is a Laurent polynomial $\Delta_K(t)\in \Z[t,t^{-1}]$ satisfying the conditions 
\begin{eqnarray}
& \Delta_K(t)=\Delta_K(t^{-1}) \\
& \Delta_K(1)=1. 
\end{eqnarray}
It is well-known that these  conditions characterize Alexander polynomials completely.
\begin{prop}[see, for example, {\cite[Theorem 8.13]{BZ}}] \label{geog} 
For any $p(t)\in \Z[t,t^{-1}]$ satisfying $(1)$ and $(2)$, there is a knot $K$ with $\Alex=p(t)$.
\end{prop}
\noindent Unfortunately, the Alexander polynomial does not separate knots very effectively
\begin{prop}\label{bot} For any $p(t)\in \Z[t,t^{-1}]$ satisfying $(1)$ and $(2)$, there exist infinitely many distinct $K$, each with  $\Alex=p(t)$.
\end{prop}
Indeed, given a knot with $\Alex=p(t)$, we can find infinitely many distinct knots by the connected sum of $K$ with distinct knots having Alexander polynomial $1$, for example, Whitehead doubles of torus knots (see Rolfsen \cite{Rolfsen1976}).  With  more work one can produce infinitely many distinct hyperbolic  (in particular, prime) knots with any possible Alexander polynomial \cite[Theorem 8.1]{Stoimenow}.  

\ons's Floer homology theory gives rise to a {\em categorification} of the Alexander polynomial  \cite{Knots, RasThesis} in 
that it provides  a collection of bigraded  abelian groups 
\[\HFKa(K)=\bigoplus_{m,a\in\Z} \HFKa_m(K,a)\]
 whose graded Euler characteristic is the Alexander polynomial:
\begin{equation} \Alex =  \sum_{a\in\Z} \Big(  \sum_{m\in\Z}  (-1)^m \mathrm{dim} \HFKa_m(K,a)\Big) \cm t^a \end{equation}
These groups are the {\em knot Floer homology} groups of $K$ (over the field $\F=\Z/2\Z$) and, taken together, they provide an example of a {\em knot homology theory} (in the sense suggested by Rasmussen \cite{Rasmussen2005}).  They were originally defined using ideas from symplectic geometry --- specifically, the theory of pseudo-holomorphic curves ---  but have since been shown to admit a combinatorial definition \cite{MOS,MOST}.  

It is natural to ask how properties $(1)$ and $(2)$ manifest in knot Floer homology and, more generally, how any property of the Alexander polynomial can be homologically interpreted (see Section \ref{sec:back} for more details in this direction).   Property $(1)$ is reflected by a symmetry among the knot Floer homology groups:
 
\begin{named}{Symmetry}[{\cite[Proposition 3.10]{Knots}}] The knot Floer homology groups satisfy
\begin{equation} \label{eq:jsym} \HFKa_m(K,a)\cong \HFKa_{m-2a}(K,-a)\end{equation}
where $m\in\Z$ is the Maslov grading and $a\in\Z$ is the Alexander grading.
\end{named}

We say a bigraded collection of groups is {\em symmetric} if it satisfies \eqref{eq:jsym}. Property (2) is somewhat more subtle, and arises from the definition of the knot Floer homology groups as the associated graded groups of a filtered chain complex which computes the Heegaard Floer homology of $S^3$.  It is expressed as follows.

\begin{named}{Canceling differential}
There is an endomorphism \[\partial_K\co \HFKa_m(K,a) \rightarrow \underset{a'<a}\oplus\HFKa_{m-1}(K,a')\]
which satisfies $\partial_K^2= 0$.  The homology of the resulting chain complex is given by:
\[ H_*(\HFKa(K),\partial_K) \cong \begin{cases}
 \F & \textrm{if} \ m=0 \\
 0 & \textrm{otherwise}.
\end{cases}\]
\end{named}

If a collection of groups can be endowed with such an endomorphism,   we say that it is equipped with a {\em canceling differential}; compare Rasmussen \cite[Section 2]{Rasmussen2005}, and see Section \ref{subsec:infinity} for more discussion.

Given these homological lifts of $(1)$ and $(2)$, one can ask about the analogues of Proposition \ref{geog} and Proposition \ref{bot}.  The purpose of this note is to address these  questions.\footnote{This terminology is borrowed from complex surface theory.}  

\begin{named}{Geography Question}
Given a symmetric, bigraded collection of abelian groups $G$ equipped with a canceling differential $\partial_G$, does there exist a knot $K$ with $(\HFKa(K),\partial_K)\simeq (G,\partial_G)$?
\end{named}

\begin{named}{Botany Question}
For  a knot $K$, how many $J$ exist with $(\HFKa(K),\partial_K)\simeq (\HFKa(J),\partial_J)$?\end{named}

\subsection{Botany} Of the two questions, botany seems more tractable.  Indeed, a few notable results indicate that the knot Floer invariants are far more faithful than the Alexander polynomial.  The first, proved by \ons\ is that if $\HFKa(K)\cong \F$ then $K$ must be the unknot \cite{GenusBounds}. Ghiggini later extended this detection to the trefoil and figure eight knots \cite{Ghiggini2007}.  Presently, these are the only knots in the $3$-sphere known to be detected by Floer homology.  Finding knots for which the botany problem has a finite answer has interesting topological ramifications. For example, the results of \ons\ combined with Ghiggini's work had as corollary Dehn surgery characterizations of the unknot, trefoils, and figure eight.  The Berge conjecture on which knots admit lens space surgeries has been translated into a finiteness conjecture for the botany problem of simple knots in lens spaces \cite{BGH,LensMe,Ras2007}, and recent work of  Li and Ni \cite{LN2013}  similarly reformulate finite filling questions in terms of botany conjectures.

In the opposite direction, it is easily seen that  knot Floer homology does not distinguish all knots.  For instance, the Floer homology of an alternating knot is determined by its Alexander polynomial and signature \cite{AltKnots} and one can easily produce distinct alternating knots sharing these invariants.  Bankwitz's theorem \cite{Bankwitz}, however, states that the number of crossings  in a reduced alternating diagram is bounded above by $|\Delta_K(-1)|$, which  implies there are only finitely many distinct alternating knots with a common Alexander polynomial (c.f.  \cite[Proposition 47]{MooreThesis}).  Thus these examples do not preclude the possibility that the answer to the botany question is always finite.   Our first result indicates that knot Floer homology, like the Alexander polynomial, is quite far from a complete knot invariant. Indeed, it says that any non-trivial band sum of two unknots gives rise to an infinite family of distinct knots with identical Floer homology.  See Figure \ref{fig:bandsum} for an example.

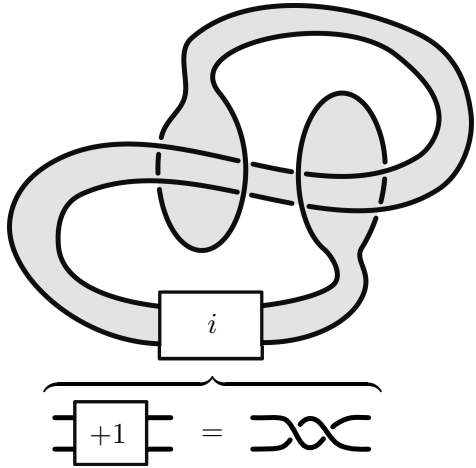
\begin{figure}[ht!] \centering
\labellist 
\pinlabel $i$ at 198 137
\pinlabel $+1$ at 97 30
\pinlabel $=$ at 198 30
\pinlabel $\overbrace{\phantom{aaaaaaaaaaaaaaaaaaaaaa}}$ at 198 73
	\endlabellist
\includegraphics[scale=0.4]{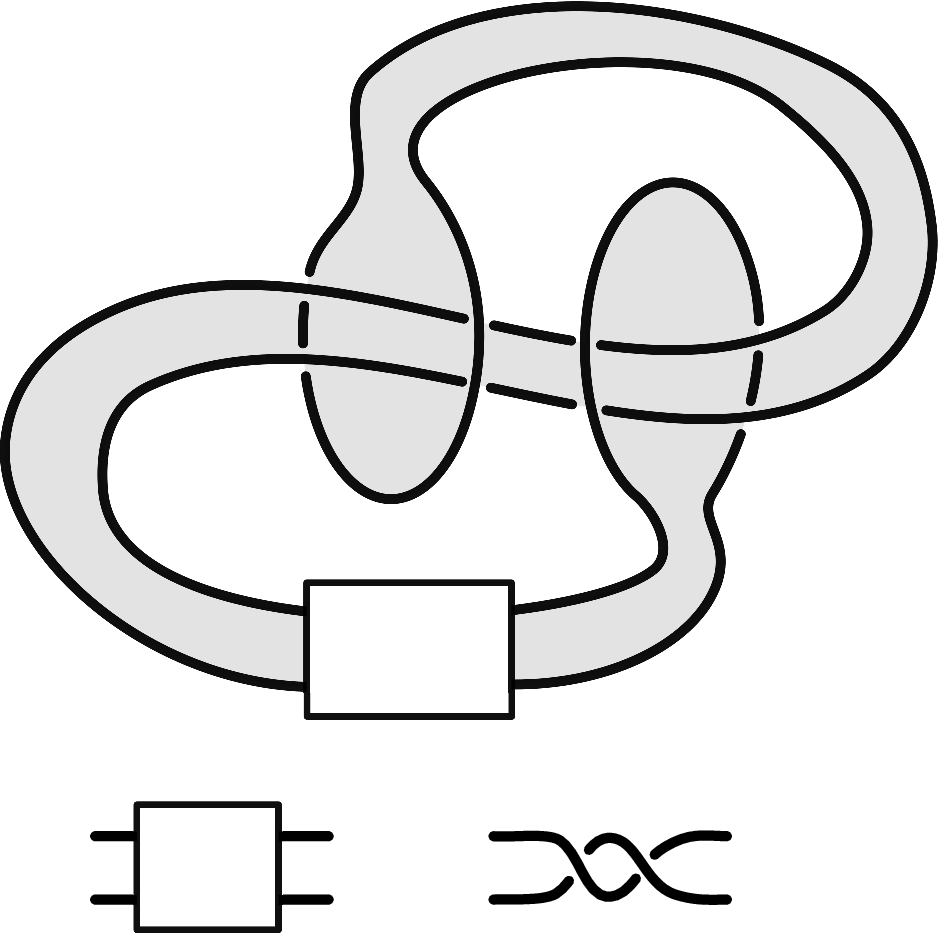}
\caption{The knot $K_i$ is a band sum of two unknots with $i$ full twists placed in the band. The template knot $K=K_0$ in this example is the knot $6_1$ in Rolfsen's table \cite{Rolfsen1976}.  The shading indicates a ribbon immersed disk bounded by  knots in the family. This disk can be resolved to an embedded disk in the $4$-ball with two radial local minima and no local maxima.}\label{fig:bandsum}
\end{figure}

\begin{thm}  \label{thm:ribbon} Let $K$ be a non-trivial knot obtained as a band sum of two unknots, and let $K_i$ be the knot obtained by adding $i$ full twists to the band, so that $K_0=K$.  Then we have
 \vspace{-3.5mm}
\begin{itemize}
\item[(i)] $\HFK(K_i)\cong\HFK(K_j)$ for all $i,j\in \bZ$; and
\item[(ii)] $K_i\not \simeq K_j$ if $i\ne j$.
\end{itemize}
\end{thm}
By $\HFK$, we mean the statement holds for both $\HFKa$ and $\HFKm$. We have the following topological corollary of Theorem \ref{thm:ribbon}:

\begin{cor}\label{cor:ribbon}
Let $K$ be a non-trivial band sum of two unknots.  Then the knots obtained by adding $i$ full twists to the band are all non-trivial and mutually distinct. Moreover, the  genus of every member in the family is the same, and if one member is fibered then they all are. \hfill\ensuremath$\Box$
\end{cor}

Non-triviality of the knots in the family follows, alternatively, from \cite[Main Theorem]{Scharlemann1985} or \cite[Theorem 1.2]{ST1989}; the latter of these results also implies that all the knots in the family have the same genus.  Note, however, that neither of these works show that the knots in our families are distinct. The tool we use to this end is Khovanov homology (see Theorem \ref{thm:Khdifferent}) --- this appears to be a new topological application of the Khovanov groups.  While it seems likely that the family obtained from any single band sum could be distinguished by other means (e.g. by hyperbolic volume), it seems difficult to handle all band sums  simultaneously.  It would be interesting if purely geometric techniques could be used to prove Corollary \ref{cor:ribbon} and, particularly, the separation result provided by Theorem \ref{thm:Khdifferent}.

The assumption that $K$ is a band sum of  two unknots is equivalent to a 4-dimensional condition; namely, that $K$ bounds a smooth and properly embedded disk in the $4$-ball on which the radius function restricts to a Morse function with no local maximum  and exactly two local minima.   In particular, any such $K$ is smoothly slice.  One may therefore be tempted to think that the lack of faithfulness of Floer homology on these knots is somehow a byproduct of their trivial concordance class  (recall that $K$ and $J$ are called {\em concordant} if there is a smooth and properly embedded cylinder in $S^3\times [0,1]$ which connects them).  However, we have following immediate corollary:

\begin{cor}\label{thm:infinite} For any knot $J\subset S^3$, there exist infinitely many distinct knots, each of which has the same concordance class as $J$ and all of which have the same knot Floer homology.\hfill\ensuremath$\Box$
\end{cor}

Indeed, a K{\"u}nneth formula  for  the knot Floer homology of the connected sum of knots \cite[Theorem 7.1]{Knots} (see  Section \ref{sec:back} for a statement), together with the prime decomposition theorem, implies that $\{J\#K_i\}_{i\in\Z}$ is an infinite family of distinct knots concordant to $J$, all of which have the same Floer homology (here $\{K_i\}_{i\in \Z}$ is  any family obtained from our theorem).  We could similarly obtain infinite families of prime knots concordant to $J$ by appealing to the known behavior of knot Floer homology under more general satellite operations \cite{Cabling, Doubling,CablingII, Hom, LOT}.   Since infinite families of knots with identical Floer invariants are so common, one may wonder what allows for the detection of the unknot, trefoil, and figure eight.  For these, and all other  detection theorems for  Floer invariants known to the authors, the key facts have been that Floer homology detects  the minimal genus of embedded surfaces in $3$-manifolds representing a given homology class \cite{NiThurston,GenusBounds}, and whether such surfaces arise as fibers in a fibration of the $3$-manifold over the circle \cite{Ghiggini2007,NiFibered, NiFibered3mfld}.  In the present context, this amounts to the fact that knot Floer homology detects both the genus of a knot and whether it is fibered.  The detection theorems now follow from the paucity of genus one (and zero) fibered knots.

Knot Floer homology also contains  geometric information related to contact structures, and importing this information yields a new detection theorem.  For a graded group we use the notation {\bf top} and {\bf bottom} respectively to indicate the maximal and minimal grading with non-trivial homology.

\begin{thm}\label{thm:identity} Suppose $K\subset Y$ is a knot with irreducible complement for which  $$\operatorname{dim}\HFKa(Y,K,\mathrm{\bf top})=\operatorname{dim}\HFKa(Y,K,\mathrm{\bf bottom})=1.$$ Suppose further that generators of these groups are non-trivial in  $H_*(\HFKa(Y,K),\partial_K)$.  Then $(Y,K)\simeq (\#^{2g}S^1\times S^2, B)$, where $B$ is the unique fibered knot of genus $g=\mathrm{\bf top}$ with monodromy isotopic to the identity, rel boundary.  
\end{thm}
\noindent This should be compared with a similar characterization theorem for $B$ due to Ni \cite[Theorem 1.3]{NiSimple}, which says  that there are  exactly $g$ distinct knots in $\#^{2g}S^1\times S^2$ having Floer homology with  rank equal to that of $B$.
By appealing to a construction of \os,  Theorem \ref{thm:identity} implies that knot Floer homology detects links with trivial monodromy.

\begin{cor}\label{cor:links} Suppose $L\subset Y$ is a fibered link whose monodromy is isotopic to the identity, rel boundary.  Then $L$ is detected by its knot Floer homology, $(\HFKa(Y,L),\partial_L)$.
\end{cor}

As a consequence we obtain a new algorithm for determining whether mapping classes are trivial, yielding a different proof of the following well-known result   (c.f. Baldwin-Grigsby \cite{BG2012} for a similar application to braid groups, and Clarkson \cite[Theorem 1]{Clarkson}.  We thank Eli Grigsby for suggesting this to us.
 
\begin{cor}[See {\cite[Theorem 4.2]{FM2012}}, or \cite{Mosher}]\label{cor:word-problem}
The mapping class group of an orientable surface with non-empty boundary has solvable word problem. 
\end{cor}

\subsection{Geography} Turning to the geography question, one could initially hope for an answer similar to that for the Alexander polynomial; namely, that {\em every} symmetric bigraded group with canceling differential can be realized as the knot Floer homology of some knot.  The reader familiar with knot Floer homology will immediately point out that  there are further restrictions on knot Floer homology groups coming from their role within the $(\Z\times\Z)$-filtered {\em infinity} version of knot Floer homology.  Thus the correct geography question should posit that the groups extend to a symmetric $(\Z\times\Z)$-filtered complex with canceling differential (see Section \ref{sec:back} for details), and one can then ask whether all groups admitting such an extension arise from some knot.   Our final result indicates that this is not the case.  It is best stated by noting that a bigraded group $(G,\partial_G)$ with canceling differential has a numerical invariant, $\tau(G)\in \Z$, defined as the minimum $a$-grading of any cycle homologous to a generator of $H_*(G,\partial_G)\cong \F$.  In the case of knot Floer homology, this is the definition of the influential \os \ concordance homomorphism $\tau(K)$ \cite{FourBall,RasThesis}.

\begin{thm}\label{thm:obstruction}
Suppose  $G=\underset{m,a\in\Z}\bigoplus G_m(a)$ is a symmetric bigraded group with canceling  differ-\myline ential $\partial_G$ such that  
\begin{itemize} 
\item[(i)] $\tau(G)=\mathrm{\bf top}$ i.e. the generator of $H_*(G,\partial_G)\cong \F$ lies in maximal $a$-grading.  
\item[(i)] $G_{-1}(\mathrm{\bf top})=G_{-1}(\mathrm{\bf top}-1)=0$
\end{itemize}
Then $(G,\partial_G)\ne (\HFKa(K),\partial_K)$ for any knot $K\subset S^3$.
\end{thm}

The theorem gives rise to a classification of knot Floer homology groups of rank $3$:  They are exactly the groups of the (right- or left-handed) trefoil.  Combined with Ghiggini's theorem we therefore have an improved detection theorem for the trefoil.

\begin{cor} \label{cor:trefoil}
If $K\subset S^3$ satisfies $\operatorname{dim}\HFKa(K)=3$ then $K$ is a trefoil.
\end{cor}

 Theorem  \ref{thm:obstruction} also leads to a new constraint on knots admitting surgeries with simple Floer homology. Recall that an L-space is a rational homology sphere $Y$ with simplest possible Heegaard Floer homology, in the sense that $\rk \HFa(Y) = |H_1(Y;\Z)|$; knots admitting non-trivial L-space surgeries are referred to as L-space knots. \ons \ showed that there are stringent restrictions on the knot Floer homology of an L-space knot \cite[Theorem 1.2]{OSz2005-lens} and these, in turn, place restrictions on the Alexander polynomial \cite[Corollary 1.3]{OSz2005-lens}; namely, the coefficients of the Alexander polynomial take values in $\{-1,0,1\}$. Combining their theorem with Theorem \ref{thm:obstruction} we have:

\begin{cor}\label{cor:L-space-knot} If $K$ is an L-space knot then the second highest Alexander grading of its knot Floer homology is non-trivial, and the  Alexander polynomial takes the form $$\Delta_K(t)= t^g - t^{g-1} \cdots - t^{1-g} + t^{-g},$$ where $g$ denotes the Seifert genus of $K$. In particular, the coefficient of $t^{g-1}$ is $-1$.  \end{cor}

\subsection{Organization}  In Section \ref{sec:back} we recall some background on knot Floer homology, calling attention to its algebraic role within the primary knot invariant from Heegaard Floer theory, the so-called {\em infinity} knot Floer complex.  We then survey  some properties of knot Floer homology, many of which will be used in the proofs that follow. 

Section \ref{sec:ribbon} proves Theorem \ref{thm:ribbon}.  This is achieved by using the skein exact sequence to verify isomorphism between the knot Floer homology groups of knots which differ by twisting along the ribbon disk (Theorem \ref{thm:HFKsame}).  We then briefly develop a similar skein exact sequence for Khovanov homology, and use it to distinguish the knots in our families in Theorem \ref{thm:Khdifferent}.  
  
Section \ref{sec:identity} proves Theorem \ref{thm:identity} by exploiting an invariant of contact structures defined using the knot Floer homology of fibered knots.  This invariant can show that a fibered knot induces a tight contact structure, and the hypotheses of the theorem imply that $K\subset Y$ induces a tight contact structure on both $Y$ and $-Y$.  We then appeal to a result of Honda, Kazez, and Mati\'c which implies that the only fibered knot $K\subset Y$ which induces a tight contact structure on both $Y$ and $-Y$ is the knot $B$ of the theorem.
 
Section \ref{sec:obstruction} proves Theorem \ref{thm:obstruction}.  The key tools for this theorem are a surgery formula relating the knot Floer homology invariants of $K\subset S^3$ to the Floer homology groups of manifolds obtained by integral surgery on $K$, and an inequality of Rasmussen relating a numerical  derivative of these latter Floer groups to the $4$-ball genus. 

Section \ref{sec:conjectures} concludes with a discussion of some conjectures and questions pertaining to the botany question, with emphasis on trying to understand the extent to which L-spaces, and the knots which give rise to them upon surgery, are detected by Floer homology.

\section{A survey of knot Floer homology} \label{sec:back}

This section provides background on the knot Floer homology invariants.  We first discuss the algebraic setting for these invariants; namely, as the associated graded groups of a $\Z$-filtered complex.  This complex, however, can be viewed as a subquotient complex of a $(\Z\times\Z)$-filtered complex, $\CFKinf(K)$, which is the primary knot invariant provided by Heegaard Floer homology.    This latter complex also allows for the definition of the {\em minus} version of knot Floer homology, which categorifies the Milnor torsion of a knot.  After discussing this, we  survey some key properties of the knot Floer invariants, especially those which will be used in the proof of our theorems.  Throughout, we use the notation $\F=\Z/2\Z$   to denote the field with two elements. 

\subsection{The infinity complex, its reduction, and derivatives}\label{subsec:infinity} We will use several variants of the knot Floer homology groups.  Each of these can be derived from  a single  invariant, the {\em infinity complex} of $K$, denoted  $(\CFKinf(K),\partial^\infty)$, which we now discuss.  

To begin,  $\CFKinf$ is a {\em graded, bilfiltered} chain complex, which means that it admits an (infinite) $\F$ basis $\mathcal{B}$ with functions:
\[ m:\mathcal{B}\rightarrow \Z \ \ \text{and}\ \ \Filt:\mathcal{B}\rightarrow \Z\times\Z   \]
called the  {\em Maslov grading} and {\em bifiltration}, respectively, which are compatible with the differential  in the sense that for any  $a,b\in\mathcal{B}$, \begin{equation}\label{m-filt}
 a \in \partial^\infty b\ \implies\ m(a)=m(b)-1\ \ \text{and} \  \  \Filt(a)\le \Filt(b).   
 \end{equation}
Here  $a \in \partial^\infty b$ means that the coefficient of $a$ appearing in the expansion of $\partial^\infty b$  is non-zero, and  $\le$ denotes the partial order on $\Z\times\Z$ given by $(i,j)\le (i',j')$ if $i\le i'$ and $j\le j'$.  Due to the first relation, the Maslov grading is frequently referred to as the homological grading.

While infinitely generated over $\F$, $\CFKinf$ is freely and finitely generated as a module over  $\F[U,U^{-1}]$ by certain collections of intersection points of curves on a Heegaard diagram, which we call {\em generators}.  Let us denote the set of  generators  by $\G$, so that the basis $\mathcal{B}$ over $\F$ is given by  elements  $U^{d}\x$, with $d\in \Z$, $\x\in \G$.   On such elements, the bifiltration is given by:
\[\Filt(U^d\x)=(-d, A(\x)-d)\]
where $A:\mathcal{G}\rightarrow \Z$ is a function defined on generators called the {\em Alexander grading} (we will, more generally, refer to the second coordinate of $\Filt$ as the Alexander grading). Thus the variable  $U$  has  bifiltration $(-1,-1)$.  The Maslov grading is also well-behaved with respect to the $\F[U,U^{-1}]$ module structure, and satisfies
 \[m(U^d\cm \alpha)= m(\alpha)-2d,\]  for any Maslov-homogenous element $\alpha\in\CFKinf$.  
 The differential on $\CFKinf$ counts certain pseudo-holomorphic disks in a symmetric product of a Heegaard diagram.\footnote{Several combinatorial interpretations of this invariant now exist \cite{BaldLev,MOS,MOST,Sing}.}
 
The relation between the bifiltration function and the differential in (\ref{m-filt}) endows $\CFKinf$ with the structure of a $(\Z\times\Z)$-filtered chain complex.  To understand this, let $S(i,j)$ denote the subgroup of $\CFKinf$ generated by basis elements in the set $\Filt^{-1}(\{\le i\}\times \{\le j\})$.  The right half of (\ref{m-filt}) implies that the $S(i,j)$ satisfy
\[ \partial^\infty S(i,j) \subseteq S(i,j), \ \ \text{and} \ \ S(i,j)\subseteq S(i',j') \ \text{if} \ (i,j)\le (i',j')\]
 i.e. that  the $S(i,j)$ are subcomplexes of $\CFKinf$ and of each other, where  inclusion of subcomplexes is governed by the partial order on $\Z\times\Z$.  Clearly the union of all $S(i,j)$ is equal to $\CFKinf$.   A complex equipped with an exhausting sequence of subcomplexes indexed by $\Z\times \Z$ in this way is, by definition, a $(\Z\times\Z)$-filtered complex.   
  The $(\Z\times\Z)$-filtered chain homotopy type of $\CFKinf(K)$ is an invariant of the knot $K$ which was discovered independently by \ons \ \cite{Knots} and Rasmussen \cite{RasThesis}.  
  
It is often convenient to regard  $\CFKinf(K)$  as a collection of basis elements (dots) arranged at integer lattice points in the plane.  Powers of  the variable $U$ act by translation along lattice points lying  along the lines of slope one. The differential can be pictured as a collection of translation invariant arrows which connect basis elements, and which travel down and to the left.  (See Figure \ref{fig:trefoil} for an illustration.) With this picture in mind, the knot Floer homology groups can be recovered from  the  vertical strip  in the plane consisting consisting  of lattice points with $i$-coordinate zero.    Put differently, the subquotient complex $\Filt^{-1}(\{0\}\times \Z)$ of $\CFKinf(K)$  inherits a $\Z$-filtration from the second coordinate of $\Filt$.  The associated graded homology groups of this filtration are the knot Floer homology groups $\HFKa_m(K,a)$, where the Alexander grading corresponds to the filtration index.   

A useful algebraic lemma allows us to consider an often much simpler complex, whilst preserving the $(\Z\times\Z
)$-filtered chain homotopy type of $\CFKinf(K)$:  given a graded, bifiltered complex $(C,\partial)$, one can find another such complex $(\overline{C},\overline{\partial})$ called the {\em reduction} of $(C,\partial)$, which is $(\Z\times\Z)$-filtered chain homotopy equivalent to $(C,\partial)$, but for which  the restriction of $\overline{\partial}$ to the subquotient complex $\Filt^{-1}(i,j)$ vanishes for any pair $(i,j)$ \cite[Lemma 4.5]{RasThesis}. Informally, $\overline{\partial}$ is zero within any given lattice point.  We state and (tersely) prove the lemma for posterity since it has not appeared in this form, though is implicitly used throughout the literature:

\begin{named}{Reduction Lemma} Let $(C,\partial,\Filt)$ be a graded, bifiltered complex, which is freely and finitely generated over $\F[U,U^{-1}]$ as above by a collection of Maslov and Alexander homogeneous generators $\mathcal{G}$, where $U$ has Maslov grading $-2$ and bilfiltration $(-1,-1)$.  Then there is another complex $(\overline{C},\overline{\partial},\overline{\Filt})$, called the {\bf reduction of $(C,\partial,\Filt)$} satisfying:
\begin{enumerate}
\item $(\overline{C},\overline{\partial})$ is $(\Z\times \Z)$-filtered chain homotopy equivalent to $(C,\partial)$ (where filtrations are induced by $\overline{\Filt}$ and $\Filt$, respectively),
\item $(\overline{C},\overline{\partial})$ is generated over $\F[U,U^{-1}]$ by a  subset  $\overline{\mathcal{G}}\subset \mathcal{G}$,
\item The restriction of $\overline{\partial}$ to $\overline{\Filt}^{-1}(i,j)$ is identically zero for any $(i,j)\in\Z\times\Z$.
\end{enumerate}
\end{named}
\begin{proof}  The proof is essentially an equivariant application of the  well-known cancellation lemma (e.g. \cite[Lemma 5.1]{RasThesis}).  More precisely, suppose that in the complex $(C,\partial)$, we have a non-zero term in the restriction of the differential to $\Filt^{-1}(0,j)$:
\[ \y\in \partial \x, \ \ \Filt(\y)=\Filt(\x)=(0,j)\]
where $\y,\x$ are elements in $\mathcal{G}$ (for convenience, we work with the subquotients with first $\Filt$-index zero so that elements from our $\F$ basis are in $\mathcal{G}$).  Then \cite[Lemma 4.1]{HeddenNi}, applied over the ring $R=\F[U,U^{-1}]$, says that we can obtain a new complex $(C',\partial')$ which is freely generated over $\F[U,U^{-1}]$ by $\mathcal{G}\smallsetminus \{\x,\y\}$ which is homotopy equivalent to $(C,\partial)$.   A bifiltered version of \cite[Lemma 4.2]{HeddenNi} implies that we can extend the bifiltration function $\Filt$ to a function $\Filt'$, and the resulting $(\Z\times \Z)$-filtered chain homotopy type is the same as that on $(C,\partial)$ induced by $\Filt$.  We now repeatedly apply the lemma, a sequence which must terminate in a complex for which the restriction of the resulting differential to $\Filt^{-1}(0,j)$ is zero (by finiteness of this subspace).  We repeat for each of the (finite number of) non-zero subquotient complexes $\Filt^{-1}(0,j')$, arriving at a filtered chain homotopy equivalent complex $(\overline{C},\overline{\partial})$ freely generated over $\F[U,U^{-1}]$ by a subset of $\mathcal{G}$, for which the restriction of $\overline{\partial}$ to each subquotient $\Filt^{-1}(0,j)$ is zero.  But this implies that the restriction of $\overline{\partial}$ to each of the subquotients  $\Filt^{-1}(i,j)$ is zero, by free generation of the complex over $\F[U,U^{-1}]$.
\end{proof}

Consider then, the reduction of $\CFKinf(K)$.   Restricting this complex to  $\Filt^{-1}(\{0\}\times \Z)$, we obtain a bigraded chain complex whose groups are isomorphic to $\HFKa(K)$, and with a differential which we denote $\partial_K$.  This allows us to think of the knot Floer homology groups as a chain complex in their own right, with a differential that strictly lowers the Alexander grading.  This is the perspective taken in the introduction; compare \cite[Section 2]{Rasmussen2005}, and see \cite[Sections 4.5 and 5.1]{RasThesis} for more details.  As a final observation,  note that  the reduced complex  is generated as an $\F[U,U^{-1}]$ module by the knot Floer homology groups,  
\begin{equation}\label{generation} \CFKinf(K)\simeq \HFKa(K)\otimes\F[U,U^{-1}]\end{equation}  Of course the differential on $(\CFKinf(K),\partial^\infty)$ is not generated by the differential on $\HFKa(K)$. In general, only the purely vertical components of $\partial^\infty$ are determined by $\partial_K$.  Our discussion now brings us to a more refined version of the geography question:

\begin{named}{Precise Geography Question}
Which $(\Z\times\Z)$-filtered chain homotopy types of  graded bifiltered  complexes arise as $\CFKinf$ complexes of knots in the three-sphere?
\end{named}
To any such complex $(G^\infty,\partial^\infty)$  one can consider its reduction and the associated  {\em hat} complex $({G},\partial_G)$ i.e. the $\Z$-filtered subquotient $\Filt^{-1}(\{0\}\times\Z))$, equipped with its induced differential.  As observed in the introduction, certainly any graded bifiltered chain complex which arises from knot Floer homology has a symmetric hat complex for which the induced differential is canceling.     {\it A priori}, however, these are not the only restrictions.  Indeed, $\CFKinf$ itself has a global canceling differential, in the sense that its homology is isomorphic to $\F[U,U^{-1}]$.  Moreover, each vertical ``slice"  is isomorphic to a shifted version of the knot Floer homology groups
\[\Filt^{-1}(\{i\}\times\Z)) = (\HFKa(K),\partial_K)[2i,i],\]
where the notation on the right means that the Maslov grading  has been shifted up by $2i$, and the Alexander grading by $i$.

In addition to the knot Floer homology groups, we will work with another derivative of $\CFKinf(K)$, the {\em minus} knot Floer homology groups.  These are the associated graded homology groups of the subcomplex $\Filt^{-1}( \Z_{\le0}\times \Z$) (the  $2^{\rm nd}$ and $3^{\rm rd}$ quadrants of the $(i,j)$-plane), again endowed with  a $\Z$-filtration  coming from the second coordinate function.    We  denote these groups $\HFKm_m(K,a)$.  Their graded Euler characteristic satisfies:
\[ \frac{\Alex}{(1-t^{-1})} =  \sum_a \Big(  \sum_m  (-1)^m \mathrm{dim} \HFKm_m(K,a)                \Big) \cm t^a\]
The minus groups inherit an $\F[U]$-module structure from the $U$-action on $\CFKinf$, and this structure determines the hat Floer homology groups through a long  exact sequence for each $a\in\Z$:
\begin{equation*}\begin{tikzpicture}[>=latex] 
\matrix (m) [matrix of math nodes, row sep=1.5em,column sep=2em]
{ \cdots & \HFKm_m(K,a) & \HFKm_{m-2}({K},a-1) &  \HFKa_{m-2}(K,a-1)  & \cdots\\};
\path[->,font=\scriptsize]
(m-1-1) edge[->] (m-1-2)
(m-1-2) edge[->] node[above] {$U$}  (m-1-3)
(m-1-3) edge[->](m-1-4)
(m-1-4) edge[->] node[above] {$\delta$} (m-1-5);
\end{tikzpicture}\end{equation*}
The connecting homomorphism $\delta$ raises both Alexander and Maslov gradings by one.

\subsection{Properties of knot Floer homology}
We now survey some important properties of the knot Floer invariants, many of which will be instrumental in the proofs of our theorems.   We begin with a lift  of  Conway's skein relation for the Alexander polynomial to knot Floer homology, the so-called skein exact sequence.  This exact sequence will be the key tool for the proof of Theorem \ref{thm:ribbon}.

\begin{named}{Skein Exact Sequence}[{\cite[Theorem 1.1]{skein}}]  
Let $\mathcal{K}_+$, $\mathcal{K}_0$, and $\mathcal{K}_-$ be three links, which differ at a single crossing as in Figure \ref{crossing}.  Suppose that the two strands meeting at the distinguished crossing in $\mathcal{K}_+$ belong to the same component, so that in the oriented resolution the two strands correspond to two components, $i$ and $j$, of $\mathcal{K}_0$.  Then there are long exact sequences 
\[\begin{tikzpicture}[>=latex] 
\matrix (m) [matrix of math nodes, row sep=1.5em,column sep=2em]
{ \cdots & \HFKa_m(\mathcal{K}_+,a) &\HFKa_m(\mathcal{K}_-,a) &  \HFKa_{m-1}(\mathcal{K}_0,a) & \cdots\\};
\path[->,font=\scriptsize]
(m-1-1) edge[->] node[above] {$\hat{h}$}(m-1-2)
(m-1-2) edge[->] node[above] {$\hat{f}$}  (m-1-3)
(m-1-3) edge[->] node[above] {$\hat{g}$} (m-1-4)
(m-1-4) edge[->] node[above] {$\hat{h}$} (m-1-5);
\end{tikzpicture}\]
\[\begin{tikzpicture}[>=latex] 
\matrix (m) [matrix of math nodes, row sep=1.5em,column sep=2em]
{ \cdots & \HFKm_m(\mathcal{K}_+,a) & \HFKm_m(\mathcal{K}_-,a) &   H_{m-1}\left(\frac{CFK^-(\mathcal{K}_0)}{U_i-U_j},a\right)& \cdots\\};
\path[->,font=\scriptsize]
(m-1-1) edge[->] node[above] {$h^-$}(m-1-2)
(m-1-2) edge[->] node[above] {$f^-$}  (m-1-3)
(m-1-3) edge[->] node[above] {$g^-$} (m-1-4)
(m-1-4) edge[->] node[above] {$h^-$} (m-1-5);
\end{tikzpicture}\]
The $h$ maps preserve both Maslov and Alexander grading. Moreover, the second sequence is equivariant with respect to the action by $U$. \end{named}

\begin{figure}[ht!] \centering
\labellist 
	\pinlabel $\Kcross_-$ at 31 -7
	\pinlabel $\Kcross_0$ at 127 -7
	\pinlabel $\Kcross_+$ at 226 -7
	\endlabellist
\includegraphics[width=2in]{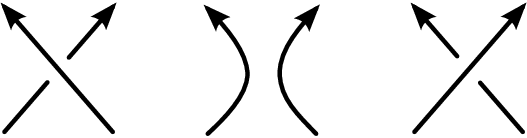}
\caption{Three links forming a skein triple}\label{crossing}
\end{figure}

There is a similar theorem for the case when the two strands belong to different components.  Since the theorem  involves links, we should recall that there are several Floer homology invariants for links.  In the above, the invariants which appear are the hat and minus  {\em knot} Floer homology groups of the link.   These are again bigraded theories, which capture the {\em single} variable Alexander polynomial of the link (as opposed to the {\em multi}-variable Alexander polynomial captured by the {\em link} Floer homology groups):
\[ (t^{1/2}-t^{-1/2})^{n-1} \cm \Delta_L(t) =  \sum_a \Big(  \sum_m  (-1)^m \mathrm{dim} \HFKa_m(L,a) \Big) \cm t^a\]
\[ \Delta_L(t) =  \sum_a \Big(  \sum_m  (-1)^m \mathrm{dim} \HFKm_m(L,a) \Big) \cm t^a.\]
In the first equation $n$ denotes the number of components of $L$.  The minus version is most naturally a module over $\F[U_1,\ldots,U_n]$, where each variable carries Alexander grading $-1$ and Maslov grading $-2$.  For our purposes the key point  about the invariants of links is that  for the two-component unlink \[H_*\Big(\frac{CFK^-(\mathrm{Unlink})}{U_1-U_2}\Big)\cong \F[U]\oplus \F[U] \]with the bi-grading of $1$ in the first summand given by $(m,a)=(0,0)$ and in the second summand by $(m,a)=(-1,0)$.  This can be calculated from directly from a genus zero, four-pointed, Heegaard diagram adapted to the unlink.

It is well known that the Alexander polynomial is insensitive to reflection and orientation reversal:
\[ \Delta_{\overline{K}}(t)=\Delta_K(t)\quad  \mathrm{and}\quad \Delta_{K^r}(t)=\Delta_{{K}}(t) \] where $\overline{K}$ is the mirror image of $K$,  and $K^r$ is $K$ with its orientation reversed.  Knot Floer homology satisfies analogous {\em categorified} versions of these equalities. 

\begin{named}{Mirror Duality}[{\cite[Proposition 3.7]{Knots}}] There is a grading-reversing isomorphism
\[(\HFKa(\overline{K}),\partial_{\overline{K}})\cong  (\HFKa(K),\partial_K)^*\]
where the term on the right is the hom dual complex.  In particular, \[\HFKa_m(\overline{K},a)\cong \HFKa_{-m}(K,-a)\] for all $a,m\in\Z$.\end{named}

\begin{named}{Reversal Insensitivity}[{\cite[Proposition 3.9]{Knots}}] There is a $(\Z\times\Z)$-filtered chain homotopy equivalence
\[\CFKinf(K^r) \simeq \CFKinf(K)\]
and, moreover, the complex on the left may be obtained from the complex on the right by composing $\Filt$ with the map defined by $(i,j)\mapsto (j,i)$.\end{named}

Reversal insensitivity places  strong restrictions on the $(\Z\times\Z)$-filtered homotopy types which can arise from knot Floer homology.  In particular, it implies that the horizontal subquotient complex $\Filt^{-1}(\Z\times \{0\})$ is  $\Z$-filtered homotopy equivalent to $\Filt^{-1}( \{0\}\times \Z)$ i.e. to knot Floer homology equipped with its canceling differential.   In fact, this induced $\Z$-filtered homotopy equivalence is responsible for the symmetry of the knot Floer homology groups mentioned in the introduction.  Indeed, we have \[\Filt^{-1}(0,j)\simeq  \Filt^{-1}(j,0)\simeq U^{j}\Filt^{-1}(0,-j)\]
where the first equivalence is induced by the one at hand, and the second follows from (\ref{generation}).  Now observe that the two ends are isomorphic to $\HFKa_*(K,j)$ (by definition) and  $\HFKa_{*-2j}(K,-j)$ (by the fact that $U^j$ has Maslov grading $-2j$), respectively.

The Alexander polynomial is well-behaved under  satellite operations.  The simplest case of this behavior is the formula for connected sums,  $\Delta_{K_1\#K_2}(t)=\Delta_{K_1}(t)\cm \Delta_{K_2}(t)$.  For knot Floer homology we have:

\begin{named}{K{\"u}nneth Formula}[{\cite[Theorem 7.1]{Knots}}]
\[ \HFKa(K_1\#K_2)\cong \HFKa(K_1)\otimes\HFKa(K_2)\]\end{named}
The above tensor product is taken in the bigraded sense, meaning that for each $m,a\in \Z$, we have \[\HFKa_m(K_1\#K_2,a)\cong \underset{m_1+m_2=m,\  a_1+a_2=a} \bigoplus\HFKa_{m_1}(K_1,a_1)\otimes\HFKa_{m_2}(K_2,a_2)\]
A similar theorem holds for $\HFKm$ but is somewhat more complicated due to the $\operatorname{Tor}$ terms which naturally arise in the context of $\F[U]$-modules.

Floer homology is also well-understood under more general satellite operations \cite{Cabling, Doubling, CablingII, Hom,Levine,Petkova, LOT}.  We discuss these results (and their implications for this work) further in Section \ref{sec:ribbon}.

The Alexander polynomial can be defined using a Seifert surface for a knot.  As such, it is not surprising that it is related to the geometry of such surfaces.  For instance, if we let $a_j$ denote the coefficient of $t^j$ in $\Delta_K(t)$, and $g(K)$ denote the Seifert genus of $K$, is easily shown that:
\[ g(K)\ge \text{deg}\ \Delta_K:=\text{max}\{ j\in \Z \ |\ a_j\ne 0\} \]
\[ a_{g(K)}=\pm 1 \ \text{if K is fibered}\]
 The information knot Floer homology provides about Seifert surfaces is considerably stronger:

\begin{named}{Genus Detection}[{\cite[Theorem 1.2]{GenusBounds}}]
\[g(K) = \mathrm{max} \{a\in\Z \ | \ \HFKa(K,a)\ne 0\}\]
\end{named}

\begin{named}{Fibered Knot Detection}[{\cite[Theorem 1.1]{NiFibered}, cf. \cite{Ghiggini2007,Juhasz}}]
\[\mathrm{rank}\ \HFKa(K,g(K))=1\ \text{if and only if} \ K\ \text{is fibered}.\]
\end{named}

Both theorems have extensions to knots in arbitrary manifolds (where the latter requires irreducibility of the knot complement).  The theorems indicate a strong connection to embedded surfaces bounded by a knot in three-space.  There is a similar connection to surfaces in four-space.  Recall from the introduction the invariant $\tau(K)$ is defined  as the minimal Alexander grading of any cycle in $(\HFKa(K),\partial_K)$ which generates the homology.  We have

\begin{named}{Four-Ball Genus Bound}[{\cite[Corollary 1.3]{FourBall}}]
Let $g_4(K)$ denote the smooth four-ball genus; that is, the minimum genus of any smooth and properly embedded surface in the four-ball, bounded by $K$. Then
\[ |\tau(K)| \le g_4(K).\]
\end{named}

There is a useful interpretation of $\tau(K)$ in terms of the module structure on $\HFKm$.  It says that $\tau(K)$ is proportional to both the Alexander or Maslov grading of the generator of a distinguished free submodule.   To state it, let $\F[U]_{\{m,a\}}$ denote the free bigraded $\F[U]$-module in which $1\in \F[U]$ has Maslov grading $m$ and Alexander grading $a$, respectively.  We have

\begin{named}{$\HFKm$ Structure Theorem}\label{prop:tau} For any knot $K$, there is a splitting of bigraded $\F[U]$-modules:
\[\HFKm(K)\cong \F[U]_{\{-2\tau(K),-\tau(K)\}}\oplus \mathrm{Tor},\]
 where Tor is a bigraded, finitely-generated, torsion $\F[U]$-module.
\end{named}


\begin{proof} 
The fact that $\operatorname{rank}_{\F[U]}\HFKm(K)=1$ follows easily from the facts that $\CFKinf$ has a canceling differential and that every element in the quotient $\CFKp:=\CFKinf/\CFKm$ is $U$-torsion.   Indeed, we have an exact sequence of  $\F[U]$ modules:
\begin{equation*}\begin{tikzpicture}[>=latex] 
\matrix (m) [matrix of math nodes, row sep=1em,column sep=2em]
{ \cdots \!\!\!\!\phantom{I} & \HFKm & \HFKinf &  \HFKp  & \phantom{I}\!\!\!\!\cdots \\
&& \F[U,U^{-1}] &&\\};
\path[->,font=\scriptsize]
(m-1-3) edge[<->] node[right] {$\cong$} (m-2-3)
(m-1-1) edge[->] (m-1-2)
(m-1-2) edge[->]  (m-1-3)
(m-1-3) edge[->](m-1-4)
(m-1-4) edge[->] node[above] {$\delta$} (m-1-5);
\end{tikzpicture}\end{equation*}
where the last term is torsion, and the first term is finitely generated as an $\F[U]$-module (by the Heegaard diagram).  Such a sequence can only exist if the rank of the first term is one, so the structure theorem for finitely generated modules over a PID gives us the splitting claimed.  Thus the heart of the theorem is to show that the element $1\in \F[U]$ has Alexander grading $-\tau(K)$ and Maslov grading $-2\tau(K)$.
This was proved in \cite[Lemma A.2]{OST2008} for the Alexander grading, but we will prove both for completeness. All complexes appearing in the proof will have differentials induced from $\CFKinf$, reduced according to Section \ref{subsec:infinity}. 

Let $\filta(a)=\Filt^{-1}(\{0\}\times \{\le a\})$ denote the filtered subcomplex of $\HFKa(K)=\Filt^{-1}(\{0\}\times \Z)$ consisting of elements with Alexander grading less than or equal to $a$.  Then an equivalent definition of $\tau(K)$ --- indeed, the original definition --- is
\[ \tau(K)=\mathrm{min}\ \{a\in \Z \ | \ \iota_*: H_*(\filta(a))\rightarrow H_*(\HFKa)\cong \F \ \text{is surjective}\}\] 

The minus Floer homology groups are, by definition:
\[\HFKm_*(K,a):=  H_*(\Filt^{-1}(\{\le 0\}\times \{a\})).\]
We now have homotopy equivalences:
\[\Filt^{-1}(\{\le 0\}\times \{a\})\simeq \Filt^{-1}( \{a\}\times \{\le 0\})
\simeq  \Filt^{-1}(\{0\}\times \{\le -a\})[-2a],\]
where the first is given by reversal insensitivity and the second by the remarks following the statement of the true geography question (the shift of $-2a$ is in the Maslov grading).  Taking homology of the extremal complexes, we have 
\begin{equation}\label{hatminus}
\HFKm_*(K,a) \cong H_{*+2a}(\filta(-a)).
\end{equation}
Recall that the $U$-module structure on $\CFKinf$, and hence $\HFKm$, is induced by the identification of groups $\CFKinf = \HFKa\otimes\F[U,U^{-1}]$.  From this, it follows that under (\ref{hatminus})  an element in $\HFKm_*(K,a)$ of infinite $U$-order corresponds to a homology class in $H_{*+2a}(\filta(-a))$ which maps onto the generator of $H(\HFKa)\cong\F$.  Conversely, a homology class in $H_{*+2a}(\filta(-a))$ mapping onto a generator gives rise to an element of infinite order in $\HFKm$.  The proposition now follows immediately, noting the reversal in Alexander grading and the shift in Maslov grading in (\ref{hatminus}).
\end{proof}

There are two more important properties of knot Floer homology which we will utilize in our proofs.  The first is a connection with contact geometry.  To state it, we first recall that with a fibered knot $K\subset Y$ one can associate an essentially unique contact structure on $Y$, denoted $\xi_K$ (see \cite{ThurstonWinkelnkemper} for the construction of $\xi_K$ and \cite{Torisu} for its uniqueness).   Knot Floer homology gives rise to an invariant of $\xi_K$ in the following sense

\begin{named}{Contact Invariance}[{\cite[Theorem 1.3]{Contact}}] Given a fibered knot $K\subset Y$ with fiber surface of genus $g$,   let $c_{\rm\bf bot}$ be a 
generator for the  non-trivial knot Floer homology group in bottommost Alexander grading  \[\HFKa(-Y,K,-g)\cong\F\langle c_{\rm\bf bot}\rangle.\]
Then the class $c(\xi_K)$  defined by \[c(\xi_K):=[c_{\rm\bf bot}]\in H_*(\HFKa(-Y,K),\partial_K)\cong \HFa(-Y),\] is an invariant of $\xi_K$, meaning that for any other fibered knot $J$ with $\xi_J\simeq\xi_K$, we have $c(\xi_J)=c(\xi_K)\in \HFa(-Y)$.
\end{named}
Strictly speaking, the Alexander grading  here depends on the relative homology class of the fiber surface for its definition.  The fact that the bottommost group has rank one is a consequence of the fact that $K$ is fibered, and the fact that $\partial_K(c_{\rm\bf bot})=0$ (so that the homology class of $c_{\rm\bf bot}$ is defined) follows from the fact that we use the reduced complex, so that there are no chains in $\HFKa$ with Alexander grading less than $-g$.

The final property of knot Floer homology used in this paper relates   the filtered homotopy type of $\CFKinf(K)$ to the Floer homology of closed $3$-manifolds obtained by surgery on $K$.  Before stating it, we first point out that to a $3$-manifold with $\SpinC$ structure $\spinc$, there are three Floer chain complexes, $\CFm(Y,\spinc),\CFinf(Y,\spinc),\CFp(Y,\spinc)$, related by a short exact sequence.  

Now let  $S^3_n(K)$ denote the $3$-manifold obtained by $n$-framed surgery on $K$, and let $-W_n$ denote the associated $4$-dimensional $2$-handle cobordism with its orientation reversed,  viewed ``backwards" as a cobordism from  $S^3_n(K)$ to $S^3$.  In terms of this cobordism we define $\spinc_m$ to be the   unique $\SpinC$ structure on $S^3_n(K)$ which extends over $-W_n$ to a  $\SpinC$ structure $\spinct_m$ with Chern class given by $c_1(\spinct_m)=(-n+2m)\cdot S$, where $S\in H^2(-W_n)\cong\Z$ is a generator.

\begin{named}{Surgery Formula}[{\cite[Theorem 4.4]{Knots}}]Let $S^3_n(K)$ denote the manifold obtained by $n$-framed surgery on $K$ and let $\spinc_m$ denote the $\SpinC$ structure defined above.  Then for all $n\ge 2g(K)-1$, and any $m$ in the interval \[\lceil (-n+1)/2 \rceil \le m \le \lfloor n/2
  \rfloor\] there is a commutative diagram of short exact sequences:
\[\begin{tikzpicture}[>=latex] 
\matrix (m) [matrix of math nodes, row sep=1.5em,column sep=1.5em]
{  0 & {\Filt^{-1}(\{i< 0\},\{j<m\})} & \CFKinf(K) & \frac{\CFKinf(K)}{\Filt^{-1}(\{i< 0\},\{j<m\})} & 0 \\
0 & \CFm(S^3_n(K),\spinc_m) & \CFinf(S^3_n(K),\spinc_m) & \CFp(S^3_n(K),\spinc_m) & 0 \\ };
\path[->,font=\scriptsize]
(m-1-1) edge[->] (m-1-2)
(m-1-2) edge[->] node[above] {$i$}(m-1-3)
(m-1-2) edge[->] node[left] {$\simeq$}(m-2-2)
(m-1-3) edge[->] node[above] {$p$}  (m-1-4)
(m-1-3) edge[->] node[left] {$\simeq$}(m-2-3)
(m-1-4) edge[->] node[above] {$$} (m-1-5)
(m-1-4) edge[->] node[left] {$\simeq$}(m-2-4)
(m-2-1) edge[->] (m-2-2)
(m-2-2) edge[->] node[above] {$i$}(m-2-3)
(m-2-3) edge[->] node[above] {$p$}(m-2-4)
(m-2-4) edge[->] node[above] {}(m-2-5);
\end{tikzpicture}\]
where  the horizontal maps are inclusion into, and projection onto,  sub and quotient complexes, respectively, and the vertical maps are chain homotopy equivalences of complexes of $\F[U]$-modules.  
\end{named}

 There is a corresponding surgery formula which computes the Floer homology for all framed surgeries along $K$ \cite[Theorem 1.1]{IntegerSurgeries} in terms of a mapping cone of complexes derived from $\CFKinf$, and a further refinement which computes the Floer homology of all (rational-sloped) Dehn surgeries \cite[Theorem 1.1]{RationalSurgeries}.  These formulae are  also very useful, but we will have no need for them in the present article.

\section{Proof of Theorem \ref{thm:ribbon}}\label{sec:ribbon}
In this section we prove Theorem \ref{thm:ribbon}, which we break into two results.  The first (Theorem \ref{thm:HFKsame}) shows that all members of a family of knots obtained by twisting along a band sum of two unknots have isomorphic knot Floer homology.  This is an application of the skein exact sequence.  The second (Theorem \ref{thm:Khdifferent}) distinguishes the  knots in such a family in the case that the band sum we start with is a non-trivial knot.  For this we use Khovanov homology and a similar exact sequence in that context (Proposition \ref{prp:skein}), together with Kronheimer and Mrowka's result that Khovanov homology detects the unknot \cite{KM2010}.  

\begin{theorem} \label{thm:HFKsame} Let be $K$ be a band sum of two unknots, and $K_i$ the knot obtained from $K$ by adding $i$ full twists along the band (see Figure \ref{fig:bandsum} for an illustration).  Then $\HFKm(K_i)\cong \HFKm(K)$ as bigraded modules over $\F[U]$.
\end{theorem}

\begin{proof}  We will show that $K_i$ and $K_{i+1}$ have isomorphic Floer homology for any $i\in\Z$.  The key observation is that $K_i$ and $K_{i+1}$ are related by a single crossing change. Further, for each $i\in\Z$,  the oriented resolution of the crossing results in the two component unlink.  Indeed, this resolution cuts the band from which $K=K_0$ is constructed.  Letting $\mathcal{K}_+=K_{i}$, $\mathcal{K}_-=K_{i+1}$, and $\mathcal{K}_0=\text{Unlink}$, we can apply the skein exact sequence to relate $\HFKm$ of the three links.

To do this, recall that for any knot the $\HFKm$ structure theorem  gives a  decomposition
\[\HFKm(K)\cong \F[U]_{\{-2\tau,-\tau\}}\oplus \operatorname{Tor}.\]
In the case at hand $K_i$ is a ribbon knot for all $i\in \Z$. Hence the smooth $4$-ball genus  of $K_i$  is zero, and the four-ball genus bound for $\tau$ implies that for each $i\in\Z$ we have
\[\HFKm(K_i)\cong \F[U]_{\{0,0\}}\oplus \mathrm{Tor}_i\]

As a result the skein exact sequence takes the form:
\[\begin{tikzpicture}[>=latex] 
\matrix (m) [matrix of math nodes, row sep=1.5em,column sep=2em]
{ \cdots & \F[U]_{\{0,0\}}\oplus \operatorname{Tor}_i &  \F[U]_{\{0,0\}}\oplus \operatorname{Tor}_{i+1} &   \F[U]_{\{0,0\}}\oplus \F[U]_{\{-1,0\}}& \cdots\\};
\path[->,font=\scriptsize]
(m-1-1) edge[->] node[above] {$h^-$}(m-1-2)
(m-1-2) edge[->] node[above] {$f^-$}  (m-1-3)
(m-1-3) edge[->] node[above] {$g^-$} (m-1-4)
(m-1-4) edge[->] node[above] {$h^-$} (m-1-5);
\end{tikzpicture}\]
where all maps are $\F[U]$-module homomorphisms.  The fact that the torsion submodules are finitely generated, together with the structure theorem for modules over a PID, implies that the torsion modules are finite dimensional as $\F$-vector spaces. Thus each torsion submodule has non-trivial elements in but a finite number of Alexander gradings.  On the other hand,  $U$ carries bidegree $(m,a)=(-2,-1)$, and so the free submodules have non-trivial elements in all Alexander gradings less than zero. It follows that for all $a\ll 0$ the exact sequence gives:
\[\begin{tikzpicture}[>=latex] 
\matrix (m) [matrix of math nodes, row sep=1.5em,column sep=1.5em]
{  & \HFKm_{2a}(\mathcal{K}_0,a) &  \HFKm_{2a}(K_i,a) &  \HFKm_{2a}(K_{i+1},a) & \HFKm_{2a-1}(\mathcal{K}_0,a) &\\
0 & \F & \F & \F & \F &0\\ };
\path[->,font=\scriptsize]
(m-1-2) edge[->] node[above] {$h^-$}(m-1-3)
(m-1-2) edge[->] node[left] {$\cong$}(m-2-2)
(m-1-3) edge[->] node[above] {$f^-$}  (m-1-4)
(m-1-3) edge[->] node[left] {$\cong$}(m-2-3)
(m-1-4) edge[->] node[above] {$g^-$} (m-1-5)
(m-1-4) edge[->] node[left] {$\cong$}(m-2-4)
(m-1-5) edge[->] node[left] {$\cong$}(m-2-5)
(m-2-1) edge[->] (m-2-2)
(m-2-2) edge[->] node[above] {$1$}(m-2-3)
(m-2-3) edge[->] node[above] {$0$}(m-2-4)
(m-2-4) edge[->] node[above] {$1$}(m-2-5)
(m-2-5) edge[->] (m-2-6);
\end{tikzpicture}\]
where each $\F$ is generated by the monomial $U^{-a}$ in one of the free $\F[U]$ summands.   Analyzing $g^-$, we have
\[U^{-a}\cm g^-(1)= g^-(U^{-a}\cm 1)= U^{-a}\in \F[U]_{\{-1,0\}}\] where $1\in \F[U]\subset \HFKm(K_{i+1})$ is the generator.  The first equality follows from $U$-equivariance, and the second is the lower right isomorphism in the diagram above. 

Thus $g^-(1)\ne 0$, and since $1\in \F[U]_{\{-1,0\}}$ is the only element in $\HFKm(\mathrm{Unlink})$ with grading $-1$, it follows that $g^-(1)=1$.  Hence, by $U$-equivariance, $g^-$ maps $\F[U]\subset \HFKm(K_{i+1})$ isomorphically onto $\F[U]_{\{-1,0\}} \subset \HFKm(\mathrm{Unlink}).$  Moreover, as the target is free, the torsion submodule is in the kernel of $g^-$.

The same analysis shows that $h^-$ maps $\F[U]_{\{0,0\}}\subset \HFKm(\mathrm{Unlink})$ isomorphically onto $\F[U]\subset \HFKm(K_{i})$;  indeed, 
\[U^{-a}\cm h^-(1)= h^-(U^{-a}\cm 1)= U^{-a}\] 
for all $a\ll0$. Thus $h^-(1)$ is an element in $\HFKm_0(K_i,0)$  which is not $U$-torsion.  The only such element is $1\in\F[U]$.   Exactness now implies that $f^-$ is a bigraded isomorphism between the torsion submodules $\operatorname{Tor}_i$ and $\operatorname{Tor}_{i+1}$. Since the free summands are isomorphic it follows that $\HFKm(K_i)\cong \HFKm(K_j)$ for all $i,j\in \Z$.
\end{proof}

Note that since knot Floer homology detects the genus and fiberedness of a knot, the above theorem implies all members of a family obtained by twisting along a band have the same genus and fiberedness status.  In particular, it follows immediately that $K_i$ is non-trivial for any integer $i$. 

Thus to prove Theorem \ref{thm:ribbon} (and Corollary \ref{cor:ribbon}), it remains to verify that the knots $K_i$ and $K_j$ are distinct for all $i\ne j$, in the case that  $K_i$ is non-trivial for some, and hence all, $i\in \Z$.   For this we will use Khovanov homology \cite{Khovanov2000}.   In order to draw as close a parallel with knot Floer homology as possible, we  employ the reduced version of Khovanov homology $\Khred(K)$  \cite{Khovanov2003}, with  $q$-grading  half of the quantum grading in \cite{Khovanov2000}.   As above, we take coefficients in $\bF=\bZ/2\bZ$.  With these conventions the Jones polynomial of a knot $K$ in $S^3$ is recovered by
\begin{equation*} V_K(t) =  \sum_{q\in\Z} \Big(  \sum_{u\in\Z}  (-1)^u \rk \Khred_q^u(K)\Big) \cm t^q, \end{equation*} 
normalized to take the value $1$ on any diagram of an unknot and satisfying the  skein relation
\begin{equation}\label{jones-skein} t^{-1}\cdot V_{\otherrightcross}(t)  -   t \cdot V_{\otherleftcross}(t) = (t^{-1/2}-t^{1/2}) \cdot V_{\otheroriented}(t).\end{equation}
We prove the following:

\begin{theorem}\label{thm:Khdifferent} Let $K$ be a non-trivial band sum of two unknots, and $K_i$ the knot obtained from $K$ by adding $i$ full twists along the band.  Then $\Khred(K_i)\ncong \Khred(K_j)$ if $i\ne j$.
\end{theorem}

For the reader content to distinguish the knots in a specific family, the skein relation for the Jones polynomial implies that \begin{equation}\label{jones}V_{K_i}(t)=t^{2(j-i)}(V_{K_j}(t)-1)+1\end{equation} for  all $i,j\in \Z$. Thus  the Jones polynomials of knots in a family will be mutually distinct, provided that none of these polynomials is trivial.\footnote{Should this case arise the polynomials are trivial for all $i\in\bZ$; this would give rise to an infinite family of non-trivial knots with trivial Jones polynomial.}  It is currently unknown whether there exists a non-trivial knot with trivial Jones polynomial.  This explains why we turn to Khovanov homology, equipped with Kronheimer and Mrowka's unknot detection theorem.  As such our proof might be viewed as a categorification of Equation \eqref{jones}.

\subsection{The skein exact sequence in Khovanov homology}  We develop the oriented exact sequence which categorifies the skein relation of Equation (\ref{jones-skein}).  This exact sequence was alluded to in Rasmussen   {\cite[Section 4.2]{Rasmussen2005}}, and proved for $sl_N$ link homology theory in \cite[Proposition 7.6]{Rasmussen}. We essentially follow this latter proof in the context of reduced Khovanov homology\footnote{Up to taking mirrors, the $N=2$ case of Khovanov and Rozansky's (reduced) $sl_N$-homology \cite{KR-1,KR-2} studied by Rasmussen in \cite{Rasmussen} coincides with reduced Khovanov homology (see also Hughes \cite{Hughes2013}).}, with conventions tailored to our needs. See Figure \ref{fig:conventions} for an illustration of the conventions used here. 

We find it convenient to use the {\em diagonal} grading $\delta=u-q$ giving rise to  a $(\Z\times \Z)$-graded (co)homological invariant $\Khred^\delta_q(K)$ of a knot $K$ (for links this invariant is $(\half\Z\times \half\Z)$-graded). A shift operator $[\cm,\cm]$ adjusts the bigrading by the rule $\Khred^\delta_q(K)[i,j]=\Khred^{\delta-i}_{q-j}(K)$. 

There are a pair of long exact sequences for this theory described by Rasmussen \cite{Rasmussen2005}, though our conventions are consistent with Manolescu and Ozsv\'ath \cite[Proposition 2.2]{MO2008}. Let $n_-(K)$ count the number of negative crossings in a fixed orientation of (a diagram of) $K$. Given a distinguished positive crossing  $\otherrightcross$ in $K$ set $c_+=n_-(\zero)-n_-(\otherrightcross)$; this constant compares the negative crossings in the original diagram with the negative crossings in a choice of orientation on the resolution (different choices of orientation may result in different constants). Similarly, set $c_-=n_-(\zero)-n_-(\otherleftcross)$ for a negative crossing. Then for each $q$ we have long exact sequences
\[\begin{tikzpicture}[>=latex] 
\matrix (m) [matrix of math nodes, row sep=1.5em,column sep=2em]
{ \cdots &\Khred(\zero)[-\half c_+,\half(3c_+ +2)]&  \Khred(\otherrightcross) &   \Khred (\otheroriented)[-\half,\half] & \cdots\\};
\path[->,font=\scriptsize]
(m-1-1) edge[->] (m-1-2)
(m-1-2) edge[->] node[above] {$i^*_+$}  (m-1-3)
(m-1-3) edge[->]  (m-1-4)
(m-1-4) edge[->] node[above] {$\partial_+^*$} (m-1-5);
\end{tikzpicture}\]
\[\begin{tikzpicture}[>=latex] 
\matrix (m) [matrix of math nodes, row sep=1.5em,column sep=2em]
{ \cdots &\Khred (\otheroriented)[\half,-\half] &  \Khred(\otherleftcross) &   \Khred(\zero)[-\half(c_-+1),\half(3c_-+1)]& \cdots\\};
\path[->,font=\scriptsize]
(m-1-1) edge[->] (m-1-2)
(m-1-2) edge[->] node[above] {$i^*_-$}  (m-1-3)
(m-1-3) edge[->] (m-1-4)
(m-1-4) edge[->] node[above] {$\partial^*_-$} (m-1-5);
\end{tikzpicture}\]
where the connecting homomorphisms $\partial^*_\pm$ raise the $\delta$-grading by one and preserve the $q$-grading, and the other maps preserve bigrading. Both long exact sequences arise from a  natural splitting at the chain level from the inclusions of subcomplexes 
\[\textstyle i_+\co \CKhred(\zero)[-\half c_+,\half(3c_+ +2)]\hookrightarrow  \CKhred(\otherrightcross) \quad{\rm and}\quad 
i_-\co \CKhred (\otheroriented)[\half,-\half] \hookrightarrow  \CKhred(\otherleftcross) \]
where $\CKhred(K)$ is the chain complex computing $\Khred(K)$. In particular, each long exact sequence arises from a mapping cone construction. Towards comparing $\Kcross_+$ and $\Kcross_-$ (as in Figure \ref{crossing}), for the same choice of orientation on the (common) resolution $\zero$ at the distinguished crossing, let $c=c_+$ so that $c_-=c-1$. Then the relevant mapping cones in this setting are  
\begin{align}
\label{positive}
\CKhred(\otherrightcross) & \textstyle= \cone{\partial_+ \co \CKhred (\otheroriented)[-\half,\half] \to \CKhred(\zero)[-\half c,\half(3c+2)]}
\\
\label{negative}
\CKhred(\otherleftcross) & \textstyle= \cone{\partial_- \co \CKhred(\zero)[-\half c,\half(3c -2)] \to \CKhred (\otheroriented)[\half,-\half] }
\end{align}

Following \cite[Section 7.3]{Rasmussen} and \cite[Section 4.2]{Rasmussen2005}, we deduce the oriented skein exact sequence from this pair of mapping cones.  Rasmussen observes that from (\ref{negative}) it follows there is a homotopy equivalence \[\begin{tikzpicture}[>=latex] 
\matrix (m) [matrix of math nodes, row sep=-0.5em,column sep=2em]
{\CKhred(\zero)[-\half c,\half(3c +2)] & \cone{i_-\co \CKhred (\otheroriented)[-\half,\frac{3}{2}]\to \CKhred(\otherleftcross)[0,2]} \\
\rotatebox[origin=c]{90}{$\in$} & \rotatebox[origin=c]{90}{$\in$} \\
x & (\partial_-(x),(x,0)) \\};
\path[->,font=\scriptsize] 
(m-1-1) edge[->]  node[above] {$\iota$} (m-1-2)
 (m-3-1) edge[|->] (m-3-2);
 \end{tikzpicture}\]
where the target vector $\iota(x)$ is written with respect to the natural direct sum decomposition of the mapping cone consisting of two copies of the complex associated with $\otheroriented$ on either side of  the complex for $\zero$.  Note we have applied an overall (and cosmetic) shift $[0,2]$, and that there is an additional shift of $-1$ in the $\delta$-grading of $\CKhred (\otheroriented)$ to ensure that $i_-$, now the differential in a mapping cone, raises $\delta$-grading by one.  In fact $\iota$ is a homotopy inverse to the natural projection from $\cone{i_-}$ onto its $\CKhred(\zero)$ summand. Moreover, $\iota$ is an inclusion in a strong deformation retract \cite[Proof of Proposition 7.6]{Rasmussen}, so by an observation of Bar-Natan \cite[Lemma 4.5]{Bar-Natan2005} we obtain 
\[\CKhred(\otherrightcross) \textstyle\simeq \cone {\iota\partial_+ \co \CKhred(\otheroriented)[-\half,\half] \to \cone{i_-} }.\] We can unpack this iterated cone description of  $\CKhred(\otherrightcross)$ as follows:
\[\begin{tikzpicture}[>=latex] 
\matrix (m) [matrix of math nodes, row sep=0em,column sep=4em]
{\CKhred(\otheroriented)[-\half,\half]& \\
& \CKhred(\otherleftcross)[0,2] \\
\CKhred(\otheroriented)[-\half,\frac{3}{2}] & \\};
\path[->,font=\scriptsize]
(m-1-1) edge[->] node[above]  {$(\partial_+,0)$} (m-2-2)
(m-1-1) edge[->] node[left] {$\partial_-\partial_+$} (m-3-1)
(m-3-1) edge[->] node[below] {$i_-$} (m-2-2);
\end{tikzpicture}\]
By construction, the maps in this three-step filtration on  $\CKhred(\otherrightcross)$ raise $\delta$-grading by 1  and  preserve the $q$-grading. 
Clearly $ \CKhred(\otherleftcross)[0,2]$ is a subcomplex, from which it follows that there is a short exact sequence
\[\begin{tikzpicture}[>=latex] 
\matrix (m) [matrix of math nodes, row sep=1.5em,column sep=2em]
{ 0 & \CKhred(\otherleftcross)[0,2]& \CKhred(\otherrightcross)  &  \cone{\partial_-\partial_+}& 0.\\};
\path[->,font=\scriptsize]
(m-1-1) edge[->] (m-1-2)
(m-1-2) edge[->]  (m-1-3)
(m-1-3) edge[->]  (m-1-4)
(m-1-4) edge[->] (m-1-5);
\end{tikzpicture}\]
Notice that taking the graded Euler characteristic yields
$$ V_{\otherrightcross}(t) = t^2\cdot V_{\otherleftcross} + \chi(\cone{\partial_-\partial_+})$$
where $\chi(\cone{\partial_-\partial_+})=(t^{1/2}-t^{3/2}) \cdot V_{\otheroriented}(t)$. This can be rewritten, after multiplication by $t^{-1}$, as
$$ t^{-1}\cdot V_{\otherrightcross}(t) - t\cdot V_{\otherleftcross}(t) = (t^{-1/2}-t^{1/2}) \cdot V_{\otheroriented}(t)$$
and compared with Equation \eqref{jones-skein}.  

It will be necessary to have an interpretation of  $\cone{\partial_-\partial_+}$, and not merely its Euler characteristic, in terms of the link obtained from the oriented resolution.  To this end, recall that choosing a marked point $p_i$ on each of the $l$ components of a link diagram endows the Khovanov complex with the structure of  a module over the ring $\F[x_1,...,x_{l}]/(x_1^2,...,x_{l}^2)$.  On the vector subspace $(\F[X]/X^2)^{\otimes n} $ in $CKh$ arising from a given complete resolution,  the variable $x_i$ acts as multiplication by $X$ on the tensor factor corresponding to the unknotted component in the resolution containing $p_i$.  We refer to the endomorphism $x_i$ as a {\em basepoint map}.  The basepoint maps are chain maps and, taken together, they endow the Khovanov complex and its homology with the aforementioned module structure.  The  module structure on homology is an invariant of the link (see \cite[Section 2]{HeddenNi}, \cite{Khovanov2003} or   \cite{Rasmussen} for more details).  

The invariance of the Khovanov module identifies $\cone{\partial_-\partial_+}$ as an invariant of the oriented resolution: one checks that if basepoints $p_1$ and $p_2$ are placed on either strand near the resolved crossing in $\otheroriented$, then we have $\partial_-\partial_+ = x_1 + x_2$, as endomorphisms of the unreduced Khovanov complex.  Picking one of these two points as distinguished, say $p_1$, the reduced complex $\CKhred$ of each of the three links in the skein relation is formed as the cokernel complex of $x_1$.  It follows that $$\cone{\partial_-\partial_+}= \cone{x_2: \CKhred({\otheroriented})\to \CKhred({\otheroriented})},$$
and the proof of invariance of the Khovanov module implies that the homology of the mapping cone on the right is an invariant of the $2$-pointed link (specifically  \cite[Lemma 2.3]{HeddenNi}, \cite[Lemma 5.16]{Rasmussen}).   In the case that  $\otheroriented$ has two components, $\cone{x_2}$ is an invariant of  $\otheroriented$ together with an ordering of its components which Rasmussen calls the  {\em totally reduced homology};  and in the case that $\otheroriented$ is a knot, it is simply a knot invariant.

\begin{proposition}[Rasmussen {\cite[Section 4.2]{Rasmussen2005} \cite[Proposition 7.6]{Rasmussen}}]\label{prp:skein}
Let $\Kcross_\pm$ be the links of Figure \ref{crossing} that differ by a single crossing change. There is a long 
exact sequence 
\[\begin{tikzpicture}[>=latex] 
\matrix (m) [matrix of math nodes, row sep=1.5em,column sep=2em]
{ \cdots & \Khred(\Kcross_-)[0,2] & \Khred(\Kcross_+)  &  H_*(\cone{\partial_-\partial_+})& \cdots\\};
\path[->,font=\scriptsize]
(m-1-1) edge[->] (m-1-2)
(m-1-2) edge[->]  (m-1-3)
(m-1-3) edge[->]  (m-1-4)
(m-1-4) edge[->] (m-1-5);
\end{tikzpicture}\]
where  $(\partial_-\partial_+)^*\co \Khred(\Kcross_0)[-\half,\half]\to\Khred(\Kcross_0)[-\half,\frac{3}{2}]$ is the basepoint map $(x_2)^*$, an invariant of $\Kcross_0$ relative to the component(s) on either side of the resolution.  All maps preserve the $q$-grading; the connecting homomorphisms raise the $\delta$-grading by 1. 
\hfill\ensuremath$\Box$
\end{proposition}

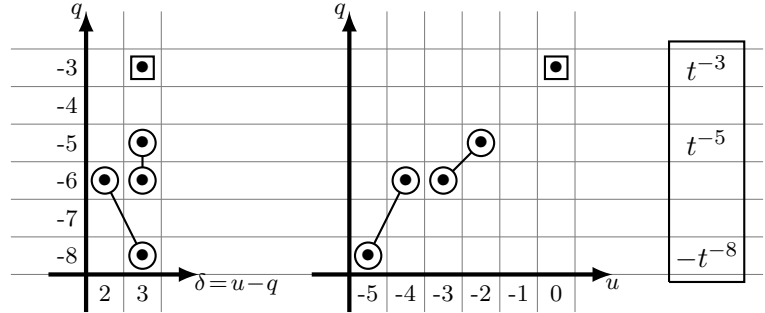
\begin{figure}[ht!]
\begin{tikzpicture}[>=latex] 

\draw [gray] (-5,3.5) -- (5,3.5);
\draw [gray] (-5,3) -- (5,3);
\draw [gray] (-5,2.5) -- (5,2.5);
\draw [gray] (-5,2) -- (5,2);
\draw [gray] (-5,1.5) -- (5,1.5);
\draw [gray] (-5,1) -- (5,1);
\draw [gray] (-5,0.5) -- (5,0.5);

\draw [gray] (2.5,0) -- (2.5,4);
\draw [gray] (2,0) -- (2,4);
\draw [gray] (1.5,0) -- (1.5,4);
\draw [gray] (1,0) -- (1,4);
\draw [gray] (0.5,0) -- (0.5,4);
\draw [gray] (0,0) -- (0,4);
\draw [gray] (-0.5,0) -- (-0.5,4);

\draw [gray] (-3,0) -- (-3,4);
\draw [gray] (-3.5,0) -- (-3.5,4);
\draw [gray] (-4,0) -- (-4,4);


\node at (4.5-1.5,-0.125+0.5) {\footnotesize{$u$}};
\node at (1-0.125-1.5,4) {\footnotesize{$q$}};

\draw[ultra thick,->] (1-1.5,0) -- (1-1.5,4);
\draw[ultra thick,->] (0.5-1.5,0.5) -- (4.5-1.5,0.5);

\node at (1.25-1.5,.25) {\footnotesize{-$5$}};
\node at (1.75-1.5,.25) {\footnotesize{-$4$}};
\node at (2.25-1.5,.25) {\footnotesize{-$3$}};
\node at (2.75-1.5,.25) {\footnotesize{-$2$}};
\node at (3.25-1.5,.25) {\footnotesize{-$1$}};
\node at (3.75-1.5,.25) {\footnotesize{$0$}};

\node at (-4.25,0.75) {\footnotesize{-$8$}};
\node at (-4.25,1.25) {\footnotesize{-$7$}};
\node at (-4.25,1.75) {\footnotesize{-$6$}};
\node at (-4.25,2.25) {\footnotesize{-$5$}};
\node at (-4.25,2.75) {\footnotesize{-$4$}};
\node at (-4.25,3.25) {\footnotesize{-$3$}};

\node at (1.75-1.5,1.75) {$\bu$};\draw[thick] (1.75-1.5,1.75) circle (0.175);

\node at (1.25-1.5,0.75) {$\bu$};\draw[thick] (1.25-1.5,0.75) circle (0.175);\draw[thick,shorten <=5pt,shorten >=5pt] (1.75-1.5,1.75) -- (1.25-1.5,0.75);
\node at (2.25-1.5,1.75) {$\bu$};\draw[thick] (2.25-1.5,1.75) circle (0.175);
\node at (2.75-1.5,2.25) {$\bu$};\draw[thick] (2.75-1.5,2.25) circle (0.175);\draw[thick,shorten <=5pt,shorten >=5pt] (2.25-1.5,1.75) -- (2.75-1.5,2.25);
\node at (3.75-1.5,3.25) {$\bu$}; \draw[thick] (-3.5+0.1+5.5,3+0.1) rectangle (-3-0.1+5.5,3.5-0.1);


\node at (-2,-0.125+0.5) {\footnotesize{$\delta\!=\!u\!-\!q$}};
\node at (-4.125,4) {\footnotesize{$q$}};

\node at (-3.25,0.25) {\footnotesize{$3$}};
\node at (-3.75,0.25) {\footnotesize{$2$}};

\draw[ultra thick,->] (-4,0) -- (-4,4);
\draw[ultra thick,->] (-4.5,0.5) -- (-2.5,0.5);

\node at (-3.75,1.75) {$\bu$};\draw[thick] (-3.75,1.75) circle (0.175);

\node at (-3.25,0.75) {$\bu$};\draw[thick] (-3.25,0.75) circle (0.175);\draw[thick,shorten <=5pt,shorten >=5pt] (-3.25,0.75) -- (-3.75,1.75);
\node at (-3.25,1.75) {$\bu$};\draw[thick] (-3.25,1.75) circle (0.175);
\node at (-3.25,2.25) {$\bu$};\draw[thick] (-3.25,2.25) circle (0.175);\draw[thick,shorten <=5pt,shorten >=5pt] (-3.25,1.75) -- (-3.25,2.25);
\node at (-3.25,3.25) {$\bu$};\draw[thick] (-3.5+0.1,3+0.1) rectangle (-3-0.1,3.5-0.1);

\node at (4.25,3.25) {$t^{-3}$};
\node at (4.25,2.25) {$t^{-5}$};
\node at (4.25,.75) {$-t^{-8}$};
\draw[black, thick] (3.75,0.4) -- (3.75,3.6) -- (4.75,3.6) -- (4.75,0.4) -- (3.75,0.4);
\end{tikzpicture}
\caption{Grading conventions illustrated for (the mirror of) the knot $8_{19}$ in Rolfsen's table \cite{Rolfsen1976}: In this work, the primary grading is $\delta$ (diagonal) and the secondary grading is $q$ (quantum) on the reduced Khovanov homology, $\Khred(8_{19})$.  Each $\bu$ denotes a copy of the vector space $\bF$. In this example $\tilde{s}(8_{19})=-3$, and the pairings given by the canceling differential have been illustrated (the spectral sequence, for $8_{19}$, collapses on the $E_2$-page since there are only 2 diagonals).  Notice that $V_{8_{19}}(t)=t^{-3}+t^{-5}-t^{-8}$ in this case; this graded Euler characteristic is recorded at the right. } \label{fig:conventions}\end{figure}

\subsection{A canceling differential in Khovanov homology} As with knot Floer homology, Khovanov homology admits a canceling differential owing to the existence of an analogue of Lee's spectral sequence for $\F$-coefficients. 

\begin{theorem}[Turner \cite{Turner2006}]\label{thm:Turner} Given a knot $K$ in $S^3$ there is a spectral sequence with $E_1\cong\Khred(K)$ and $E_\infty\cong\F$ supported in grading $u=\delta+q=0$. The differential $\boldsymbol d_i$ on the $E_i$-page raises the $\delta$-grading by $1-i$ and raises the $q$-grading by $i$. \hfill\ensuremath$\Box$\end{theorem}

Note that these conventions follow \cite[Section 3.4]{Watson2010}. In direct analogy with Rasmussen's $s$-invariant (for Khovanov homology over $\Q$) or the  $\tau$-invariant (for $\HFK$), the canceling differential produces an  invariant, $\tilde{s}(K)\in\Z$, defined to be the quantum grading of the generator surviving Turner's spectral sequence. We will make use of:

\begin{theorem}[Rasmussen \cite{Rasmussen2010}, Lipshitz-Sarkar \cite{LS2012}] 
For any knot $K$, $|\tilde{s}(K)|\le g_4(K).$
\hfill\ensuremath$\Box$
\end{theorem}

Note  that $2\tilde{s}(K)$ and $s(K)$ are not equal for all $K$ (the factor of 2 here is an artifact of our grading convention). Indeed, Seed \cite{Seed} found examples on which the invariants differ, e.g. $K=14^n_{19265}$ has $s(K)=0$ and $\tilde{s}(K)=-1$  \cite[Proof of Theorem 3 and Remark 6.1]{LS2012}.

With this background on Khovanov homology in place, we complete the proof of Theorem \ref{thm:ribbon} and, at the same time, Corollary \ref{cor:ribbon}.  

\begin{proof}[Proof of Theorem \ref{thm:Khdifferent}] Suppose we are given a non-trivial band sum of unknots, $K$, and let $K_i$ denote the knot obtained from $K$ by adding $i$ full twists to the band. Letting $\mathcal{K}_+=K_{i}$, $\mathcal{K}_-=K_{i+1}$, and $\mathcal{K}_0=\text{Unlink}$, we see that the oriented skein exact sequence of Proposition \ref{prp:skein} will be of use once $H_*(\cone{\partial_-\partial_+})$ is calculated in the case of $\mathcal{K}_0=\text{Unlink}$. This can be easily obtained, for instance, by applying the exact sequence to the situation when $\Kcross_+$ and $\Kcross_-$ are unknots with a single positive and negative crossing, respectively. Since $\Khred(\Kcross_+)\cong\bF_{q=0}^{\delta=0}\cong\Khred(\Kcross_-)$ in this case, it follows that 
\[H_*(\cone{\partial_-\partial_+}) \cong \F_{q=0}^{\delta=0} \oplus \F_{q=2}^{\delta={\text -}1}\] where $\Kcross_0$ is the two-component unlink.  

Thus, in the present setting the oriented skein exact sequence becomes   
\[\begin{tikzpicture}[>=latex] 
\matrix (m) [matrix of math nodes, row sep=1.5em,column sep=2em]
{ \cdots & \Khred^{\delta}_{q-2}(K_{i+1}) & \Khred^\delta_q(K_i)  &   \F_{q=0}^{\delta=0} \oplus \F_{q=2}^{\delta={\text -}1}& \Khred^{\delta+1}_{q-2}(K_{i+1}) & \cdots,\\};
\path[->,font=\scriptsize]
(m-1-1) edge[->] (m-1-2)
(m-1-2) edge[->]  (m-1-3)
(m-1-3) edge[->]  (m-1-4)
(m-1-4) edge[->] node[above] {$d$} (m-1-5)
(m-1-5) edge[->] (m-1-6);
\end{tikzpicture}\]
so that $\Khred_{q-2}^{\delta}(K_{i+1}) \cong \Khred_q^\delta(K_i)$ for $q\ne0,2$ or for $\delta\ne -1, 0,1$. This will amount to the key observation; compare Equation (\ref{jones}). More precisely, we have two exact sequences of interest:

\[\begin{tikzpicture}[>=latex] 
\matrix (m) [matrix of math nodes, row sep=1em,column sep=2em]
{0 & \Khred^{-1}_{0}(K_{i+1}) & \Khred^{-1}_2(K_i)  &   \F& \Khred^{0}_{0}(K_{i+1}) & \Khred^{0}_2(K_i) &0\\
0 & \Khred^{0}_{-2}(K_{i+1}) & \Khred^{0}_0(K_i)  &   \F& \Khred^{1}_{-2}(K_{i+1}) & \Khred^{1}_0(K_i) &0\\
};
\path[->,font=\scriptsize]
(m-1-1) edge[->] (m-1-2)
(m-1-2) edge[->]  (m-1-3)
(m-1-3) edge[->]  (m-1-4)
(m-1-4) edge[->] node[above] {$d^2$} (m-1-5)
(m-1-5) edge[->] (m-1-6)
(m-1-6) edge[->] (m-1-7)
(m-2-1) edge[->] (m-2-2)
(m-2-2) edge[->]  (m-2-3)
(m-2-3) edge[->]  (m-2-4)
(m-2-4) edge[->] node[above] {$d^0$} (m-2-5)
(m-2-5) edge[->] (m-2-6)
(m-2-6) edge[->] (m-2-7);
\end{tikzpicture}\]

For all other values of $\delta,q$ the exact sequence gives isomorphism of Khovanov homology groups. Namely, if there is a $\delta\ne -1,0,1$ supporting non-trivial Khovanov homology then for some $q$ we have 
\[\Khred^\delta_q(K_i)\cong \Khred^\delta_{q-2}(K_{i+1})\ncong0\] and, more generally, 
\[\Khred^\delta_q(K_i)\cong \Khred^\delta_{q-2(j-i)}(K_{j})\ncong0\] for all $i,j$. Hence $K_i\simeq K_j$ if and only if $i = j$, since otherwise the respective (finite dimensional) Khovanov homologies differ as graded vector spaces.  

Similarly, if there is a odd integer $q$ supporting non-trivial Khovanov homology then for some $\delta$ we have \[\Khred^\delta_q(K_i)\cong \Khred^\delta_{q-2(j-i)}(K_{j})\ncong0\] for all $i,j$ and the knots $K_i$ and $K_j$ are separated as above.  

To complete the proof then it remains to carefully analyze the case wherein $\Khred(K_i)$ is supported entirely in gradings $\delta=-1,0,1$ and only in (a finite number of)  even gradings $q$. Taking the additional structure of Khovanov homology into account (namely, the canceling differential of Theorem \ref{thm:Turner}) places further constraints on the support of Khovanov homology in this case. For example, since the differential $\boldsymbol d_i$ on the $E_i$-page raises $\delta$-grading by $1-i$, it follows that in the present setting the spectral sequence must collapse on the $E_3$-page for dimension reasons. Indeed, there are at most 3 adjacent $\delta$-gradings supporting $\Khred(K_i)$ ($\delta=-1,0,1$). Furthermore, since the differential $\boldsymbol d_i$ on the $E_i$-page raises $q$-grading by $i$, both the $\boldsymbol d_1$ and the  $\boldsymbol d_3$ differential must vanish: $\Khred(K_i)$ is supported only in even $q$-gradings. To summarize, when  $\Khred(K_i)$ is supported only in gradings $\delta=-1,0,1$ and only in even $q$-gradings, $\boldsymbol d_2$ constitutes the entire canceling differential.

Since $K_i$ is a ribbon knot, hence slice, it must be that $\Khred^0_0(K_i)\ncong 0$ so that $\tilde{s}(K_i)=0$; from the preceding discussion, $H_*\big(\Khred(K_i),\boldsymbol d_2\big)\cong \F^{\delta = 0}_{q=0}$. Moreover, as we have shown that $K_i$ is necessarily non-trivial (appealing to the fact that knot Floer homology detects the unknot; compare Theorem \ref{thm:HFKsame}), it must be that $\dim\Khred(K_i)>1$ (and odd) by applying Kronheimer and Mrowka's detection theorem for Khovanov homology  \cite[Theorem 1.1]{KM2010}. Therefore,
\[\Khred(K_i)\cong \F_{q=0}^{\delta=0}\oplus
\Big(\bigoplus_{\ell\in2\bZ}(\F^{n_\ell})_{q=\ell}^{\delta=-1}\oplus(\F^{n_\ell})_{q=\ell-2}^{\delta=0}
\oplus(\F^{m_\ell})_{q=\ell}^{\delta=0}\oplus(\F^{m_\ell})_{q=\ell-2}^{\delta=1}\Big)\] 
where all but finitely many of the integers $m_\ell,n_\ell$ are 0 (but at least one such is non-zero since $K_i$ is non-trivial). Notice that the domain of $\boldsymbol d_2$ is $\bigoplus_{\ell\in2\bZ}(\F^{n_\ell})_{q=\ell-2}^{\delta=0}\oplus(\F^{m_\ell})_{q=\ell-2}^{\delta=1}$ and the image of $\boldsymbol d_2$ is $\bigoplus_{\ell\in2\bZ}(\F^{n_\ell})_{q=\ell}^{\delta=-1}\oplus(\F^{m_\ell})_{q=\ell}^{\delta=0}$ so that the $n_\ell$ and the $m_\ell$ pair up to cancel all but the vector space $ \F_{q=0}^{\delta=0}$. With this in place, we claim that 
\[\Khred(K_j)\cong \F_{q=0}^{\delta=0}\oplus
\Big(\bigoplus_{\ell\in2\bZ}(\F^{n_\ell})_{q=\ell-2(j-i)}^{\delta=-1}\oplus(\F^{n_\ell})_{q=\ell-2-2(j-i)}^{\delta=0}
\oplus(\F^{m_\ell})_{q=\ell-2(j-i)}^{\delta=0}\oplus(\F^{m_\ell})_{q=\ell-2-2(j-i)}^{\delta=1}\Big)\] 
for all integers $j$. 

Since the long exact sequence induces isomorphisms $\Khred^\delta_q(K_i)\cong\Khred^\delta_{q-2}(K_{i+1})$ away from $q=0,2$ --- in agreement with the claim --- it suffices to lighten notation and consider the case  
\[
\Khred(K_i)\cong \F_{q=0}^{\delta=0}\oplus (\F^{n})_{q=2}^{\delta=-1}\oplus(\F^{n})_{q=0}^{\delta=0}
\oplus(\F^{m})_{q=2}^{\delta=0}\oplus(\F^{m})_{q=0}^{\delta=1}
\]
for at least one of $n$ or $m$ non-zero. Now the exact sequence gives 
\[\begin{tikzpicture}[>=latex] 
\matrix (m) [matrix of math nodes, row sep=1em,column sep=2em]
{0& \Khred^{-1}_{0}(K_{i+1})&\F^n&   \F& \Khred^{0}_{0}(K_{i+1}) & \F^m &0\\
};
\path[->,font=\scriptsize]
(m-1-1) edge[->] (m-1-2)
(m-1-2) edge[->]  (m-1-3)
(m-1-3) edge[->]  (m-1-4)
(m-1-4) edge[->]  node[above] {$d^2$}(m-1-5)
(m-1-5) edge[->]  (m-1-6)
(m-1-6) edge[->]  (m-1-7)
;
\end{tikzpicture}\]
and
\[\begin{tikzpicture}[>=latex] 
\matrix (m) [matrix of math nodes, row sep=1em,column sep=2em]
{0& \Khred^{0}_{-2}(K_{i+1})&\F^{n+1}&   \F& \Khred^{1}_{-2}(K_{i+1}) & \F^m &0\\
};
\path[->,font=\scriptsize]
(m-1-1) edge[->] (m-1-2)
(m-1-2) edge[->]  (m-1-3)
(m-1-3) edge[->]  (m-1-4)
(m-1-4) edge[->]  node[above] {$d^0$}(m-1-5)
(m-1-5) edge[->]  (m-1-6)
(m-1-6) edge[->]  (m-1-7)
;
\end{tikzpicture}\]
leading to 4 cases to consider. 

When $\rk(d^0)=1$ and $\rk(d^2)=0$ we have that 
\[\Khred(K_{i+1})\cong  \F_{q=0}^{\delta=0} \oplus
 (\F^{n-1})_{q=0}^{\delta=-1}\oplus(\F^{n+1})_{q=-2}^{\delta=0}
\oplus
(\F^{m-1})_{q=0}^{\delta=0}\oplus(\F^{m+1})_{q=-2}^{\delta=1}\]
which is impossible: With this form $H_*\big(\Khred(K_{i+1}),\boldsymbol d_2\big)\cong  \F^3$, a contradiction (note that this case does not arise when $n=0$). 

When $\rk(d^0)=1$ and $\rk(d^2)=1$ we calculate that 
\[\Khred(K_{i+1})\cong  \F_{q=0}^{\delta=0} \oplus
(\F^{n})_{q=0}^{\delta=-1}\oplus(\F^{n+1})_{q=-2}^{\delta=0}
\oplus
(\F^{m})_{q=0}^{\delta=0}\oplus(\F^{m+1})_{q=-2}^{\delta=1}
\]
forcing $H_*\big(\Khred(K_{i+1}),\boldsymbol d_2\big)\cong\F^{\delta=0}_{q=-2}$. This has the appropriate dimension, but is supported in $0\ne u=\delta+q=-2$, a contradiction. 

Similarly, when $\rk(d^0)=0$ and $\rk(d^2)=0$ we calculate that 
\[\Khred(K_{i+1})\cong  \F_{q=0}^{\delta=0} \oplus
(\F^{n-1})_{q=0}^{\delta=-1}\oplus(\F^{n})_{q=-2}^{\delta=0}
\oplus
(\F^{m-1})_{q=0}^{\delta=0}\oplus(\F^{m})_{q=-2}^{\delta=1}
\]
so that $H_*\big(\Khred(K_{i+1}),\boldsymbol d_2\big)\cong\F^{\delta=0}_{q=-2}$, again, a contradiction (note that this case does not arise when $n=0$). 

Finally, the remaining case when $\rk(d^0)=0$ and $\rk(d^2)=1$ gives 
\[\Khred(K_{i+1})\cong  \F_{q=0}^{\delta=0} \oplus
(\F^{n})_{q=0}^{\delta=-1}\oplus(\F^{n})_{q=-2}^{\delta=0}
\oplus
(\F^{m})_{q=0}^{\delta=0}\oplus(\F^{m})_{q=-2}^{\delta=1}
\]
so that $\tilde{s}(K_{i+1})=0$, as required, and the result as claimed in the case $j=i+1$. 

With this observation in hand the long exact sequence may be iterated (with  $\rk(d^0)=0$ and $\rk(d^2)=1$) to obtain the general statement, from which it follows immediately that   $\Khred(K_i)\cong\Khred(K_j)$ as graded vector spaces if and only if $i=j$.\end{proof}

\section{Proof of Theorem \ref{thm:identity}}\label{sec:identity}

Suppose that we have a knot $K\subset Y$ with irreducible complement satisfying
$$\mathrm{rank}\ \HFKa(Y,K,\mathrm{\bf top})=\mathrm{rank}\ \HFKa(Y,K,\mathrm{\bf bottom})=1.$$
Then the fibered knot detection theorem \cite[Theorem 1.1]{NiFibered} implies that $K$ is fibered, with fiber surface of genus equal  $g=\mathrm{\bf top}$.  We can now employ the contact invariance of the Floer class associated with the contact structure $\xi_K$ induced by the fibration on $Y\smallsetminus K$ \cite{ThurstonWinkelnkemper}. Recall that if $c_{\rm\bf bot}$ is a 
generator for \[\HFKa(-Y,K,{\rm \bf bottom})\cong\F\langle c_{\rm\bf bot}\rangle,\]
then the contact invariant $c(\xi_K)$ is defined by $$c(\xi_K):=[c_{\rm\bf bot}]\in H_*(\HFKa(-Y,K),\partial_K)\cong \HFa(-Y).$$  Using  mirror duality, the hypothesis that the generator of the top group represents a non-trivial Floer homology class in $\HFa(Y)$ implies that $c(\xi_K)\ne 0\in \HFa(-Y)$.  Similarly, we can consider the mirror $\overline{K}$; that is, regard $K\subset -Y$.  Now the mirror duality property, together with our assumption that the generator of the bottom group is non-trivial in $\HFa(Y)$, implies $c(\xi_{\overline{K}})\ne 0\in \HFa(Y)$.  

Thus the fibered knot $K$ induces contact structures $\xi_K, \xi_{\overline{K}}$ on $Y$ and $-Y$, respectively, each of whose \os\ contact elements is    non-trivial.  By \cite[Theorem 1.4]{Contact}, this implies each contact structure is {\em tight}.  Now observe  that reversing the orientation of $Y$ can be achieved by reversing the orientation on the page of the open book decomposition induced by $K$.  This implies that if $\phi\in \mathrm{MCG(\Sigma_{g,1})}$ is the mapping class element specifying the monodromy for the open book decomposition induced by $K$, then $\phi^{-1}$ represents the monodromy of the open book induced by $\overline{K}$.  

We now appeal to a result of Honda, Kazez and Mati\'c \cite[Theorem 1.1]{HKMVeerI} which states that a contact structure  is tight if and only if every open book decomposition supporting it has {\em right-veering} monodromy.  It follows immediately from the definition that the only element $\phi\in \mathrm{MCG(\Sigma_{g,1})}$ such that  $\phi$ and $\phi^{-1}$ are both right-veering is the trivial mapping class.  Thus the monodromy of $K$ is isotopic rel boundary to the identity.    

To finish the argument,  we observe that the complement of the binding of the open book with trivial monodromy is homeomorphic to $S^1\times \Sigma_{g,1}\cong \#^{2g}S^1\times S^2\smallsetminus B$ and (as noted in \cite{NiSimple}) that work of Gabai shows that knots in $\#^{2g}S^1\times S^2$ are determined by their complement   \cite[Corollary 2.14]{Gabai6}. \hfill$\ensuremath\Box$\ {\em Theorem \ref{thm:identity}}

\begin{proof}[Proof of Corollary \ref{cor:links}]\ons \ show that with a link $L\subset Y$ of $|L|$ components, one can associate a knot $\kappa(L)\subset Y\#^{|L|-1}S^1\times S^2$ \cite[Section 2.1]{Knots}. 
  This construction served as their original definition of the knot Floer homology of $L$:
  $$ (\HFKa(Y,L),\partial_L):= (\HFKa(Y\#^{|L|-1}S^1\times S^2,\kappa(L)),\partial_{\kappa(L)}).$$
From their definition, it is easy to see that $\kappa(L)$ is obtained from $L$ by plumbing $|L|-1$ copies of the unique fibered link with identity monodromy and fiber surface an annulus to the fiber surface for $L$, in such a way as to make the boundary of the resulting surface connected (c.f. \cite[Lemma 4.4]{NiSutured}).  Gabai showed that any plumbing of fiber surfaces is a fiber surface  \cite{MR718138}.  Moreover, the monodromy of the plumbings' fibration is the composition of those of the plumbands.  Thus, if $L$ is a fibered link with genus $g$ fiber surface and  monodromy isotopic to the identity, then $\kappa(L)$ is the  fibered knot of genus $g+|L|-1$ with identity monodromy.  The previous theorem says that $\kappa(L)$ is detected by its knot Floer homology.  Hence $L$ is detected by its knot Floer homology, by the definition of the knot Floer homology of a link.\end{proof}
Before turning to the proof of Corollary \ref{cor:word-problem}, recall that a finitely generated group $G$ has solvable word problem if, given a product of generators $w$ (a {\em word} in the generators of $G$), there exists an algorithm to decide if $w$ represents the trivial element in $G$. 

\begin{proof}[Proof of Corollary \ref{cor:word-problem}]
Let $\mathrm{MCG(\Sigma_{g,k},\partial)}$ denote the mapping class group of  diffeomorphisms of a genus $g$ surface with $k$ boundary components which fix the boundary pointwise.  Given $\phi\in\mathrm{MCG(\Sigma_{g,k},\partial)}$, let $Y=Y(\phi)$ denote the $3$-manifold specified by the  open book decomposition associated with $\phi$, and let  $L$ denote the binding. By Corollary \ref{cor:links} the knot Floer homology groups $\HFKa(Y\#^{|L|-1}S^1\times S^2,\kappa(L))$, together with their canceling differential $\partial_{\kappa(L)}$, will certify whether $\phi$ is trivial, so it remains to verify that there is an algorithm for computing this information.  This is provided by the Sarkar-Wang algorithm for computing the {\em hat} Floer homology groups of an arbitrary $3$-manifold, as well as the filtered homotopy type of the filtration of the {\em hat} complex induced by an arbitrary knot therein \cite{SW2010}. This algorithm may be adapted to the setting at hand, namely, for a  mapping class $\phi$ expressed as a word in a generating set of Dehn twists --- see Plamenevskaya \cite{MR2350279} for an approach catered to the relevant contact-geometric information used by Theorem \ref{thm:identity}.  Using the reduction process, this yields the knot Floer homology groups of $\kappa(K)$ and the canceling differential.  
\end{proof}

\section{Proof of Theorem \ref{thm:obstruction}} \label{sec:obstruction}

Noting that the genus detection of knot Floer homology \cite[Theorem 1.2]{GenusBounds} identifies the Seifert genus  with the top-most Alexander grading in the support of knot Floer homology, suppose we are given $K\subset S^3$ for which the knot Floer homology groups satisfy:
\begin{enumerate}
\item[(i)]   $\tau(K)=g$
\item[(ii)] $\HFKa_{-1}(K, g)=\HFKa_{-1}(K,g-1)=0,$
\end{enumerate}
where $g$ is the Seifert genus.  We wish to show that no such $K$ exists.  

We will use the surgery formula for knot Floer homology, together with a result of Rasmussen.    Associated with an oriented rational homology $3$-sphere $Y$ equipped with a $\SpinC$ structure $\spinc$ is the \os \ ``correction term" \cite[Section 4]{AbsGrad} (see below for the definition).   Denoted $d(Y,\spinc)$, this is a $\Q$-valued invariant derived from the $\F[U]$-module structure on the Heegaard Floer homology of $Y$ (see Fr\o yshov \cite{FroyshovD} for the Seiberg-Witten motivation for these invariants).   In the case at hand, let $d(S^3_n(K), \spinc_m)$ denote the correction term associated with the $\SpinC$ structure $\spinc_m$ on $S^3_n(K)$, which was defined before the statement of the surgery formula in Section \ref{sec:back}.

We can compare the $d$-invariants for surgeries on a knot $K$ to the $d$-invariants for the corresponding surgery on the unknot by  defining:
\[ \overline{h}_m(K):= \frac{d(S^3_n(\mathrm{Unknot}),\spinc_m)-d(S^3_n(K),\spinc_m)}{2}\]
This is essentially an invariant  defined by Rasmussen \cite{RasThesis,RasGoda} in analogy to an invariant of  Fr{\o}yshov from instanton homology \cite{FroyshovH}.   It has a $4$-dimensional interpretation as the rank of the kernel of the map on $\HFp$ induced by the $\SpinC$ $2$-handle cobordism $(-W'_n,\spinct_m)$, restricted to the image of $\HFinf$, though this interpretation will not be needed in our discussion.

The careful reader will note that $\overline{h}_m(K)$ differs from Rasmussen's invariant, given by \cite[Equation (2)]{RasGoda}:  
\[h_m(K)=\frac{d(S^3_{-n}(K),\spinc_m)- d(S^3_{-n}(\mathrm{Unknot}),\spinc_m)}{2}\]
(the unknot term is denoted $E(n,m)$ in \cite{RasGoda}).    We claim that  $$\overline{h}_m(K)=h_m(\overline{K}).$$  To see this, note first that $S^3_n(K)=-S^3_{-n}(\overline{K})$ and that $\overline{\mathrm{Unknot}}=\mathrm{Unknot}$.  Now recall that the $d$-invariants reverse sign when the orientation of a  $3$-manifold is reversed \cite[Proposition 4.2]{AbsGrad}:
$$d(-Y,\spinc)=-d(Y,\spinc).$$
The stated relationship follows immediately.  The reader may worry that the labeling conventions for $\SpinC$-structures on $S^3_n(K)$ and $S^3_{-n}(K)$ disagree.  However, it is easy to see that they are forced to agree up to the sign of $m$. Since the resulting invariants $\overline{h}_m$ and $h_m$ are independent of the sign of $m$ (by conjugation invariance of Floer homology),  we make no effort to make this technicality precise.  

A key tool is an inequality satisfied by $h_m$, due to Rasmussen.  This inequality can be viewed as a Heegaard Floer analogue of a result for instanton homology proved by Fr\o shov \cite{Froyshov}.  To state it, let $K\subset S^3$, and let
  $g_4(K)$ be its smooth $4$-ball genus. Then \cite[Theorem 2.3]{RasGoda} states that $h_m(K) = 0$ for $|m|\geq
  g_4(K)$, while for $|m|< g_4(K)$ we have
\begin{equation*}
h_m(K) \leq \Bigr\lceil \frac{g_4(K) - |m|}{2} \Bigl\rceil.
\end{equation*}
 
In the case at hand, we have observed that a knot with the putative Floer homology will have $\tau(K)=g(K)$ which, when combined with the four-ball genus bound \cite[Corollary 1.3]{FourBall}  $|\tau(K)|\le g_4(K)$  shows that $g_4(K)=g(K)$.   Since the $4$-genus is  invariant under taking mirrors,  Rasmussen's bound for $h_m$ shows that 
$$ h_{g-2}(\overline{K})\le 1.$$
We will show that our assumptions on the knot Floer homology groups, together with the surgery formula, implies  $\overline{h}_{g-2}(K)\ge 2$, thus arriving at a contradiction to prove the theorem.

\vskip0.1in

\begin{figure}
\begin{tikzpicture}[>=latex] 
\draw[lightgray, fill=lightgray] (0,0) -- (0,1.5) -- (-2.5,1.5) -- (-2.5,3.75) -- (4.25,3.75) -- (4.25,-2) -- (0,-2) -- (0,0);

\draw[lightgray, pattern=north east lines]  (-2.5,1.5) -- (-1.5,1.5) -- ( 0.75,3.75) -- (-2.5, 3.75) ;

\node at (1.5,4) (a) {\phantom{\large$\bullet$}};

\node at (1,3.5) (1) {\large$\bullet$}; 
\draw[red, thick] (a) edge[out=170,in=80,->] (1);

\node at (.5,3) (2) {\large$\bullet$}; 
\draw[thick, red] (1) edge[out=170,in=80,->] (2);

\node at (0,2.5) (3) {\large$\bullet$}; 

\draw[thick, red] (2) edge[out=170,in=80,->] (3);

\node at (-0.5,2.0) (4) {\large$\bullet$}; 
\draw[thick, red] (3) edge[out=170,in=80,->] (4);

\node at (-1,1.5) (5)  {\large$\bullet$}; 
\draw[thick, red] (4) edge[out=170,in=80,->] (5);

\node at (-1.1,-2.2) {\footnotesize{$-2$}};
\node at (.5,-1.6) {\footnotesize{$j$-axis}};

\draw[thick, dashed] (-1,-2) -- (-1,3.5);
\node at (4.625,0) {\footnotesize{$0$}};
\node at (3.4,-.2) {\footnotesize{$i$-axis}};

\node at (-0.00,-2.2) {\footnotesize{$0$}};
\node at (4.75,2.0) {\footnotesize{$g-1$}};
\node at (4.75,2.5) {\footnotesize{$g$}};
\draw[thick, dashed] (-2.5,2.0) -- (4.25,2.0);
\draw[thick, dashed] (-2.5,2.5) -- (4.25,2.5);
\node at (4.75,1.5) {\footnotesize{$g-2$}};
\draw[thick, dashed] (-2.5,1.5) -- (4.25,1.5);
\draw[ultra thick,->] (0,-2) -- (0,4);
\draw[ultra thick,->] (-2.5,0) -- (4.5,0);
\end{tikzpicture}
\caption{The shaded region in the figure indicates the portion of the $(i,j)$-plane specifying the ``hook" complex $A_{g-2}$.  The dots represent  non-trivial homology classes which exist, due to the assumption $\tau(K)=g$,  in the $E_1$ term of the spectral sequence associated with the horizontal filtration of $A_{g-2}$ (by the $i$-coordinate), and the red arrows represent the endomorphism $U$.  The hatching indicates a region  of $\CFKinf$ without generators. The left-most dot has Maslov grading $-4$. }\label{fig:hook}\end{figure}
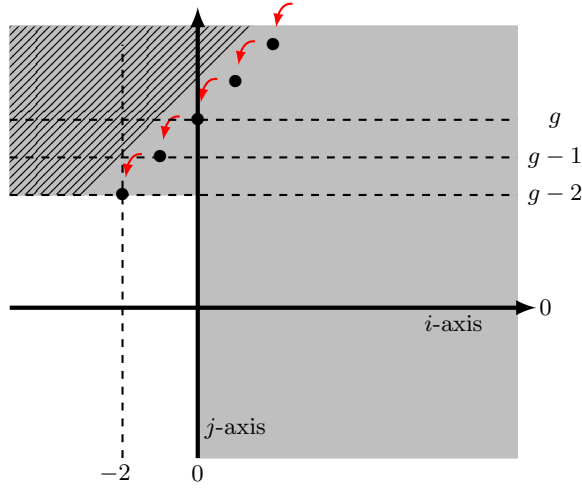

Let $A_{g-2}$ denote the quotient complex of $\CFKinf(K)$ given by \[ A_{g-2}:= \frac{\CFKinf(K)}{\Filt^{-1}(\{i< 0\},\{j<g-2\})}\] i.e. the complex generated by basis elements whose bifiltration coordinates satisfy the  constraint $\mathrm{max}(i,j-g+2)\ge0$. Geometrically, we think of this quotient complex as a ``hook" in $\CFKinf(K)$, as shown in Figure \ref{fig:hook}.   The surgery formula  \cite[Theorem 4.4]{Knots} states that for  $n\ge 2g(K)-1$, the Floer homology of $\HFp(S^3_n(K),\spinc_{g-2})$ is the homology of this quotient:  
\begin{equation}\label{eq:surgery}
\HFp_{*+s}(S^3_n(K),\spinc_{g-2})\cong H_{*}( A_{g-2}),
\end{equation}
where the shift $s$ in the Maslov grading on $\CFKinf(K)$  is independent of the knot $K$ (it depends only on the surgery coefficient $n$ and the Chern class of the extension of the $\SpinC$ structure over the 2-handle cobordism).     The correction term is defined via the $\F[U]$ module structure in a dual  manner to our characterization of $\tau(K)$ in terms of the module structure on $\HFKm$:
$$ d(Y,\spinc):= \underset{\xi \in \HFp(Y,\spinc)}{\mathrm{min}}\{ \mathrm{gr}(\xi)\  |\ \xi \in \mathrm{Im}\ U^{d} \ \forall d>0 \}.$$
Thus, according to \eqref{eq:surgery}, we have 
$$ d(S^3_n(K),\spinc_{g-2})= s+ \underset{\xi \in H_{*}(A_{g-2})}{\mathrm{min}}\{ \mathrm{gr}(\xi)\  |\ \xi \in \mathrm{Im}\ U^{d} \ \forall d>0 \}. $$
For the unknot, the second term is easily seen to be zero.  Thus $\overline{h}_{g-2}$ is given by:
$$  -2 \overline{h}_{g-2}(K)= \underset{\xi \in H_{*}( A_{g-2})}{\mathrm{min}}\{ \mathrm{gr}(\xi)\  |\ \xi \in \mathrm{Im}\ U^{d} \ \forall d>0 \}. $$
It suffices to show that the right-hand quantity is less than or equal to $-4$.  To do this, note that the quotient complex $A_{g-2}$ inherits a $(\Z\times\Z)$-filtration from $\CFKinf(K)$.  Using the $\Z$-filtration coming from the $i$-coordinate, in particular, we obtain a spectral sequence of $\F[U]$-modules.   Now consider a cycle representative for $H_*(\HFKa(K),\partial_K)\cong\F$.  Any such cycle lives in Alexander grading $\tau(K)=g$, by assumption (i), and the  $U^k$ translates of  this cycle   generate a non-trivial homology class in each filtration level $i\ge -2$ in the $E_1$ page of the spectral sequence. See Figure \ref{fig:hook}. Moreover, the class in filtration $-2$ has grading $-4$ and is in the image of $U^d$ for all $d$.   Call this class $\alpha$.     

Now observe that for each filtration $i\ge 0$, we have $E^i_1\cong \F_{(2i)}$, where the superscript indicates the filtration level. This happens  because $H_*(\HFKa(K),\partial_K)\cong \F$, and for each $i\ge 0$ the sub-quotient of $A_{g-2}$ with fixed $i$ is isomorphic to $(\HFKa(K),\partial_K)[2i,i])$. Thus the only way for $\alpha$ to die in the spectral sequence is if it is  the boundary of a chain in $E^{-1}_1$ under the $d_1$ differential.   Assumption (ii), however, implies that there are  no chains in $E^{-1}_1$ (or even in $E^{-1}_0$) with Maslov grading $-3$.  Thus $\alpha$  generates a non-trivial homology class in $H_*(A_{g-2})$ which is in the image of $U^d$ for all $d$. \hfill$\ensuremath\Box$\ {\em Theorem \ref{thm:obstruction}} 

The utility of Theorem \ref{thm:obstruction} is illustrated through an example in Figure \ref{fig:trefoil}.    

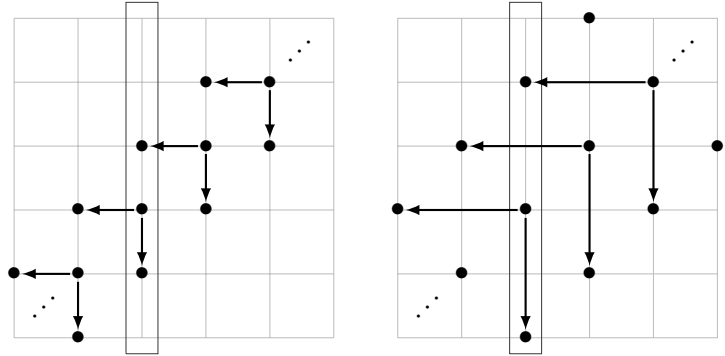
\begin{figure}[ht!]\begin{tikzpicture}[scale=0.85,>=latex]
\draw[step=1,lightgray] (0,0) grid (5,5);
\draw[darkgray] (1.75,-0.25) -- (1.75,5.25) -- (2.25,5.25) -- (2.25,-0.25) -- (1.75,-0.25);
\draw[step=1,lightgray] (6,0) grid (11,5);
\draw[darkgray] (7.75,-0.25) -- (7.75,5.25) -- (8.25,5.25) -- (8.25,-0.25) -- (7.75,-0.25);
\node at (0,1) {$\bullet$}; 
\node at (1,0) {$\bullet$}; 
\node at (1,1) {$\bullet$}; 
\node at (1,2) {$\bullet$}; 
\node at (2,1) {$\bullet$}; 
\node at (2,2) {$\bullet$}; 
\node at (2,3) {$\bullet$}; 
\node at (3,2) {$\bullet$}; 
\node at (3,3) {$\bullet$}; 
\node at (3,4) {$\bullet$}; 
\node at (4,3) {$\bullet$}; 
\node at (4,4) {$\bullet$}; 
\node at (4.5,4.5) {\rotatebox{45}{$\cdots$}};
\node at (0.5,0.5) {\rotatebox{45}{$\cdots$}};
\draw[thick,->,shorten <=3pt,shorten >=3pt] (1,1) -- (0,1);\draw[thick,->,shorten <=3pt,shorten >=3pt] (1,1) -- (1,0);
\draw[thick,->,shorten <=3pt,shorten >=3pt] (2,2) -- (1,2);\draw[thick,->,shorten <=3pt,shorten >=3pt] (2,2) -- (2,1);
\draw[thick,->,shorten <=3pt,shorten >=3pt] (3,3) -- (2,3);\draw[thick,->,shorten <=3pt,shorten >=3pt] (3,3) -- (3,2);
\draw[thick,->,shorten <=3pt,shorten >=3pt] (4,4) -- (3,4);\draw[thick,->,shorten <=3pt,shorten >=3pt] (4,4) -- (4,3);
\node at (6.5,0.5) {\rotatebox{45}{$\cdots$}};
\node at (7,1) {$\bullet$}; 
\node at (6,2) {$\bullet$}; \node at (8,2) {$\bullet$};  \node at (8,0) {$\bullet$}; 
\node at (7,3) {$\bullet$}; \node at (9,3) {$\bullet$};  \node at (9,1) {$\bullet$}; 
\node at (8,4) {$\bullet$}; \node at (10,4) {$\bullet$};  \node at (10,2) {$\bullet$}; 
\node at (9,5) {$\bullet$}; \node at (11,3) {$\bullet$}; 
\draw[thick,->,shorten <=3pt,shorten >=3pt] (8,2) -- (6,2);\draw[thick,->,shorten <=3pt,shorten >=3pt] (8,2) -- (8,0);
\draw[thick,->,shorten <=3pt,shorten >=3pt] (9,3) -- (7,3);\draw[thick,->,shorten <=3pt,shorten >=3pt] (9,3) -- (9,1);
\draw[thick,->,shorten <=3pt,shorten >=3pt] (10,4) -- (8,4);\draw[thick,->,shorten <=3pt,shorten >=3pt] (10,4) -- (10,2);
\node at (10.5,4.5) {\rotatebox{45}{$\cdots$}};
\end{tikzpicture}\caption{The knot Floer homology of the right-hand trefoil (left), and a complex that is ruled out as the knot Floer homology of any knot by Theorem \ref{thm:obstruction} (right). The $i=0$ slice has been singled out; each $\bullet$ denotes a copy of $\F$.}\label{fig:trefoil}\end{figure}

\begin{proof}[Proof of Corollary \ref{cor:trefoil}] Assume  $\mathrm{rank}\ \HFKa(K)=3$.   If $\HFKa(K,i)=0$ for all $i\ne 0$, then the  genus of $K$ is zero, and $K$ is the unknot (which has rank one Floer homology).  If $\HFKa_*(K,i)\ne 0$ for some $i\ne 0$, then symmetry implies $\HFKa_{*-2i}(K,-i)\ne 0$.  This accounts for rank $2$.  Thus symmetry implies $\HFKa(K,0)\ne 0$.  Since a canceling differential lowers Maslov grading by one and  $H_*(\HFKa,\partial)$ is supported in grading zero,  the only possibilities for the knot Floer homology groups are 
\[\HFKa_m(K,a)\cong \begin{cases}
\F & m=0,a=i\\
\F & m=1-2i,a=0\\
\F & m=-2i,a=-i\\
0& {\rm otherwise}
\end{cases}
\qquad{\rm or}\qquad
\HFKa_m(K,a)\cong \begin{cases}
\F & m=2i,a=i\\
\F & m=2i-1,a=0\\
\F & m=0,a=-i\\
0& {\rm otherwise}
\end{cases}\]	
In the first case, the only Alexander grading that supports non-trivial Floer homology in Maslov grading zero is $a=i$, hence $\tau(K)=i$.  Now if $i\ne 1$, Theorem \ref{thm:obstruction} tells us $K$ cannot exist; indeed, in this case there is no Floer homology at all in Maslov grading $-1$.  If $i=1$ Ghiggini's theorem \cite[Corollary 1.5]{Ghiggini2007} tells us $K$ is the right-handed trefoil.  Mirror duality implies that if $K$ has knot Floer homology groups as on the right, then $\overline{K}$ has the groups on the left, thus proving the corollary.\end{proof}

\begin{proof}[Proof of Corollary \ref{cor:L-space-knot}] Let $K$ be an L-space knot.  Ozsv\'ath and Szab\'o proved that that $\tau(K)=g$ \cite[Corollary 1.6]{OSz2005-lens}.   This implies that assumption $(i)$ in the statement of Theorem \ref{thm:obstruction} is satisfied and that $\rk\HFKa_0(K,g)>0$. Indeed, in order for $\tau(K)$ to equal the genus, there must be a cycle for $\partial_K$ in the top group of knot Floer homology, and this cycle must lie in Maslov grading zero as it generates the homology $H_*(\HFKa,\partial_K)$.  Moreover, they showed that $\rk\HFKa(K,a)$ is 1 or 0 for every $a$ \ \cite[Theorem 1.1]{OSz2005-lens}.  It follows that the top group of knot Floer homology is supported entirely in grading zero.   Theorem \ref{thm:obstruction} then implies that $\HFKa_{-1}(K,g-1)= 1$. The  corollary follows.
\end{proof}

\section{Questions, conjectures and concluding remarks}\label{sec:conjectures}

\subsection{Abstract infinity complexes}\label{sub:infinity} 
The results of this article, while interesting  --- at least to the authors --- in their own right, serve to highlight how little is know about the geography and botany questions in knot Floer homology. The geography question, in particular, seems quite difficult.  Let us define 

\begin{defn}  An  {\bf (abstract) infinity complex} is a  graded, bifiltered complex $(C,\partial,\Filt)$ satisfying  \begin{enumerate}
\item $(C,\partial)$ is freely generated as complex of modules over  $\F[U,U^{-1}]$ by a finite set of graded, bifiltered homogeneous generators.

\item Acting by   $U$ shifts the grading by $-2$ and the bifiltration by $(-1,-1)$.

\item   $(C,\partial)$ has a global canceling differential, in the sense that $H_*(C,\partial)\cong \F[U,U^{-1}]$, where $1\in \F[U,U^{-1}]$ has grading $0$.
\item The complex $(C,\partial, \Filt^r)$, where $\Filt^r$ is the bifiltration function $\Filt^r(i,j):=\Filt(j,i),$ is $(\Z\times\Z)$-filtered homotopy equivalent to $(C,\partial,\Filt)$
\end{enumerate}
\end{defn}
The discussion of Section \ref{sec:back} indicates that $\CFKinf(K)$ is an abstract infinity complex for any $K$.  Theorem \ref{thm:obstruction} indicates, however, that the most naive guess at an answer to the geography question  ---   that any infinity complex arises as a knot Floer infinity complex  ---  is wrong.  While our theorem places new restrictions on which infinity complexes arise, it offers little insight as to what a general characterization of knot Floer complexes could look like.  Indeed,  the theorem seems to take us further from a conjectural characterization than where we started.    

A potentially more tractable geography question could be phrased in terms of a $4$-dimensional relation placed on knot Floer complexes.  In \cite{HomCable}, Hom defined a $\{-1,0,1\}$-valued invariant of an infinity complex, $C$.  Denoted $\epsilon(C)$, the invariant measures how a generator of  the homology of the vertical complex  interacts with the horizontal components of the differential.  The behavior of $\epsilon$ under duality and tensor products  gives rise to a group $\mathcal{CFK}$, which we call the {\em knot Floer concordance group}, whose elements are $\epsilon$-equivalence classes of infinity complexes. Here, two such complexes $C,D$ are equivalent if $\epsilon(C\otimes D^*)=0$ (with $D^*$  the dual complex), with  the  product and inverse operations in $\mathcal{CFK}$  given by tensor product and duality, respectively \cite{HomCTS,HomGroup}. The primary utility of the knot Floer concordance group  is that it is  the receptor of a natural homomorphism, which we denote $H$,  from the smooth concordance group $\mathcal{C}$. Two  questions  we find interesting are:
\begin{question}  What is $\mathcal{CFK}$?  More precisely, can $\mathcal{CFK}$ be given a presentation in terms of generators and relations?
\end{question}

\begin{question}[$\mathcal{CFK}$ Geography Question] What is the image of Hom's homomorphism $H:\mathcal{C}\rightarrow \mathcal{CFK}$?  In particular, is $H$ surjective?
\end{question}
Work of \cite{HomNxZ,HomCTS} implies that  $\mathcal{CFK}$ (and indeed the image of $H$) is necessarily infinitely generated, and  has a rich filtration structure by Archimedean equivalence classes.  While difficult, there is more hope that both questions could be answered due to the algebraic flexibility that $\epsilon$-equivalence provides, and the abundance of knots whose infinity complexes are now understood.  For instance, the $\epsilon$-equivalence class of the ``shifted trefoil" complex in Figure \ref{fig:trefoil} which is obstructed by Theorem \ref{thm:obstruction}, is realized  by $H([T_{4,5}]-[T_{2,3; 2,5}])$, where $T_{2,3;2,5}$ denotes the $(2,5)$ cable of the trefoil knot \cite{HomPersonal}.   Thus one could retain hope that $H$ is surjective and that the geography question --- up to $\epsilon$-equivalence --- is answered by the naive guess: abstract infinity complexes.

\subsection{Botany and the Berge conjecture}\label{sub:bot-berge} From a topological perspective, the botany problem seems more interesting. Understanding when the question has a finite answer, in particular, seems well-motivated by potential applications to the study of Dehn surgery.  Such applications arise from the following general strategy:
\begin{enumerate}
\item Assume surgery on an (unknown) knot $K$ produces a particular manifold, $Y$ (or a particular type of manifold, e.g. a  lens space).
\item Use the surgery formula to show  this assumption implies $\CFKinf(K)$ is of a particular form.
\item Invoke a finite answer to the botany problem for such $\CFKinf$ to determine  that $K$ is a member of some finite set.
\end{enumerate}

    For instance, this strategy can be used to show that if there is an orientation-preserving diffeomorphism $S^3_n(K)\cong S^3_n(\mathrm{Unknot})$, then $K$ is unknotted;\footnote{This theorem was first proved using a similar strategy for monopole Floer homology, which at the time lacked knot Floer homology \cite{KMOS} cf. \cite{GenusBounds}. It can be reproved using knot Floer homology, as suggested.} that is, the unknot is characterized by any member of its integer surgery spectrum.   Similar results hold for the trefoil and figure eight knots \cite{Trefoil} using the same strategy and the fact that, like the unknot, these knots are determined by their knot Floer homology.  
    
 In a similar vein, the Berge Conjecture asserts that the set of knots in $S^3$ on which one can perform surgery to obtain a lens space is exactly the class which can be placed on the genus two Heegaard surface in such a way that they are {\em doubly primitive}, see \cite[Problem 1.78]{Kirby}, \cite{Berge}.    This conjecture would be implied by an affirmative answer to the following botany conjecture for knot Floer homology of knots in lens spaces:
 \begin{conj}[{\cite[Conjecture 1.5]{BGH}} cf. \cite{Ras2007,LensMe}]  Suppose a knot $K$ in the lens space $L(p,q)$ satisfies  \[\mathrm{dim}\ \HFKa(L(p,q),K)=p.\] Then $K$  is isotopic to the union of two properly embedded chords in a pair of minimally intersecting meridional  disks of the two Heegaard solid tori in $L(p,q)$. 
 \end{conj}
 
Note that there is a canceling differential $\partial_K$ on $\HFKa(L(p,q),K)$, whose homology has dimension $p$.  The dimension assumption of the theorem is therefore equivalent to $\partial_K\equiv 0$.  In the case of $S^3$, such an assumption implies that $K$ is unknotted. There was a similar  conjecture that a pair of knots  in $L(p,q)$, each of which has Floer homology of rank $p+2$, are detected by knot Floer homology \cite[Conjecture 1.6]{BGH} \cite[Figure 3]{LensMe}  (this is the lens space analogue of Corollary \ref{cor:trefoil}), however this was disproved by Baker and Hoffman \cite[Theorem 1]{BakerHoffman}.


Returning to the botany question for knots in the $3$-sphere, Theorem \ref{thm:ribbon} suggests that the knot Floer homology of most knots will be realized infinitely often.  Indeed, Theorem \ref{thm:ribbon}, combined with the behavior of knot Floer homology under satellite operations, can be used to quite flexibly produce infinite families of distinct knots with identical knot Floer homology groups (or infinity chain complexes)  having various prescribed qualities.  For instance, there are infinite families of ``thin" knots with the same Floer homology i.e. knots whose hat Floer homology groups plotted in the Maslov-Alexander plane are supported on a single diagonal.  Indeed, any knot of the form $B\# \overline{B}^r$ where $B$ is $2$-bridge will be thin and a member of $\mathcal{R}$, hence will produce an infinite family of thin knots with the same Floer homology by Theorem \ref{thm:ribbon}.  Note that by  \cite[Second paragraph of pg. 246]{AltKnots}, the filtered homotopy type of $(\CFKinf,\partial^\infty)$ of a thin knot is determined by the knot Floer homology groups  (cf. \cite[Theorem 4]{Petkova}), so that these families actually have the same infinity complexes. One can similarly produce examples of families with arbitrary $\tau$, families all of whose members are fibered (see, for example, Theorem \ref{thm:fail} and the accompanying discussion in Section \ref{sub:Kanenobu}), families whose Floer homology has arbitrary width, and families whose Alexander polynomial is arbitrarily prescribed, by appealing to the K{\"u}nneth formula and formulas for  the knot Floer homology of Whitehead doubles or cables.  

One class of knots whose Floer homology we find difficult to replicate by our constructions are the fibered knots which induce the standard tight contact structure on the 3-sphere. Combining the fibered knot detection of Floer homology and \cite[Proposition 2.1]{Hedden2010}  this class can be defined Floer theoretically as the class of knots satisfying 
\[\HFKa(K,g)=\F\ \text{supported in Maslov grading} \ 0, \ \ \text{and} \ \tau(K)=  g,\]
where $g$ is the Seifert genus.  This class is also equivalent, again by \cite[Proposition 2.1]{Hedden2010} and Ni's fibered knot detection theorem, to the class of {\em strongly quasipositive} fibered knots.  Our efforts make the following question seem natural:
\begin{question} Are there infinitely many distinct strongly quasipositive fibered knots with the same knot Floer homology?
\end{question}
If not, then this would be one class of knots which knot Floer homology  coarsely detects.\footnote{We'll say an invariant {\em coarsely detects} a topological object $X$ if only finitely many objects have the same invariant as $X$.}  While we expect, however, that there are infinitely many strongly quasipositive fibered knots with the same Floer homology, one way to approach a coarse detection theorem  would be through the following simpler question:
\begin{question} Are there infinitely many distinct strongly quasipositive fibered knots with the same Alexander polynomial?
\end{question}
In fact the answer to this question is yes, as was pointed out to us by Sebastian Baader; compare \cite{Baader}.   One can plumb a positive Hopf band to the fiber surface of the positive $(2,6)$-torus link using infinitely many distinct isotopy classes of proper arcs, all of which are homologous.  This can produce, for instance, infinitely many distinct fibered strongly quasipositive knots with Alexander polynomial equal to that of the $(2,7)$-torus knot. 

\subsection{Coarse detection of L-space knots}\label{sub:coarse-L-space} While we are pessimistic that knot Floer homology coarsely  detects strongly quasipositive fibered knots, one might be hopeful that it does so for a very interesting subset of these, the so-called L-space knots.  Recall that a knot is a (positive) L-space knot if $S^3_n(K)$ is an L-space for some $n>0$.  The simplicity of the Floer homology of L-spaces implies, by the surgery formula, that the knot Floer homology of L-space knots is tightly constrained; see Table \ref{tab:L}.  Indeed, this is the context where the strategy outlined above is likely to be most fruitful, but where   botany results  are  lacking.  

\begin{table}[ht]
\begin{tabular}{ll} 
\toprule
   {\sc Property} & {\sc  Attribution}\\
   \midrule 
    {\small$g(K)=g_4(K)=\tau(K)=s(K)$ }& Various authors, see {\small  \cite{Hedden2010}} \\[2pt]
    {\small The coefficients of $\Delta_K(t)$ take values in $\{-1,0,1\}$} & {\small Ozsv\'ath and Szab\'o \cite[Theorem 1.2]{OSz2005-lens}}\\[2pt]
    {\small $K$ is fibered} & {\small Ghiggini \cite{Ghiggini2007}, Ni \cite{NiFibered}} \\[2pt]
    {\small $K$ is strongly quasipositive} & {\small  \cite[Proposition 2.1]{Hedden2010}} \\[2pt]
    {\small $K$ induces the standard tight contact structure on $S^3$} & {\small  \cite[Proposition 2.1]{Hedden2010}}\\[2pt]
    {\small $\rk\HFKa(K,a)=1$ for $a=g,g-1$}  & {\small Theorem \ref{thm:obstruction}; Corollary \ref{cor:L-space-knot}}\\
  \bottomrule\medskip
\end{tabular}\caption{Some properties of L-space knots.}\label{tab:L}\end{table}

We make the following conjecture:

\begin{conj}\label{finiteLspace} Let $K$ be an L-space knot.  Then there are only finitely many other knots whose Floer homology is isomorphic to $\HFKa(K)$.
\end{conj} 

It should be pointed out that the structure of $\HFKa$ of an L-space knot implies that the {\em hat} Floer homology groups determine $\CFKinf$, up to homotopy, much in the same way that the infinity complex of a thin knot is determined by its knot Floer homology groups. If true, the conjecture  would have as corollary that there are only finitely many knots on which a fixed lens space can be obtained via Dehn surgery; this corollary is also implied by the Berge Conjecture.   Of relevance here is the fact that there are known pairs of knots which admit L-space surgeries and share the same Floer homology, the simplest  being   the $(2,3)$-cable of the trefoil and the $(3,4)$-torus knot \cite{CablingII}.

Notice that Conjecture \ref{finiteLspace} would follow were it known that any L-space knot could be represented as the closure of a positive braid. This is not the case, however, and we are grateful to Cameron Gordon, Jen Hom and Tye Lidman for pointing out that the $(2,3)$-cable of the trefoil is an L-space knot that cannot be represented as the closure of a positive braid. Indeed,  a result of Cromwell gives $4g$ as upper bound on the number of crossings in a diagram of a fibered, positive knot, where $g$ is the genus of the braid closure \cite[Corollary 5.1]{Cromwell1989}; compare \cite[Section 3.1]{Stoimenow2002}. In this example then we obtain an upper bound of 12, but a direct check of the knot tables certifies that the minimal crossing number of this cable (let alone a purported positive braid diagram) exceeds this bound. 

Work of the first author alluded to above implies that any L-space knot can be represented as the closure of  a strongly quasipositive braid \cite[Corollary 1.4]{Hedden2010}. By plumbing positive Hopf bands one can show that there are infinitely many distinct strongly quasipositive fibered knots of any genus greater than one, in contrast to the finiteness of positive braid closures with bounded genus. As such, Conjecture \ref{finiteLspace} places the class of L-space knots in a perhaps interesting tension between the classes of positive and quasipositive braids. 

An affirmative answer to the following two conjectures     could be helpful in understanding the botany problem for L-space knots  (see \cite{BakerMoore} for several  other interesting conjectures, and a nice discussion of  L-space knots):
\begin{conj}
Any L-space knot admits a strong inversion.
\end{conj}
{\em Update:} This has been disproved by Baker and Luecke \cite{BakerLuecke}. See also \cite[Conjecture 30]{WatsonSymm} and the accompanying discussion. 
\begin{conj}[{\cite[Conjecture 4]{LidmanMoore}}] An L-space knot contains no essential Conway spheres in its complement.
\end{conj}

While we don't really know whether to expect coarse detection of L-space knots by knot Floer homology  (we conjecture it mainly because we find it  difficult even to produce families of L-space knots, let alone families with the same Floer homology), it is perhaps more justified to conjecture such a result for particular subsets of L-space knots.  For instance, Li and Ni make the following conjecture:

\begin{conj}[{\cite[Conjecture 1.3]{LN2013}}]  Suppose $K$ has the Floer homology of an L-space knot whose Alexander polynomial has its roots on the unit circle.  Then $K$ is an iterated torus knot.
\end{conj}

The conjecture is motivated, in part, because an affirmative answer would imply a conjecture of Boyer and Zhang  that any finite filling of a hyperbolic knot complement must be of integral slope \cite{BZ1996}. The example of the $(2,3)$-cable of the trefoil and the $(3,4)$-torus knot  tells us that it is not true that Floer homology detects iterated torus knots on the nose.   Coarse detection, however, seems plausible.

\subsection{On small rank in knot Floer homology}\label{sub:rank} As with the trefoil, the figure eight --- being the other genus one fibered knot in $S^3$ --- is also characterized by knot Floer homology \cite{Ghiggini2007}. However, in contrast with Corollary \ref{cor:trefoil}, it is not clear that there is a finite collection of knots with knot Floer homology of rank 5 (both the  $(2,5)$- and $(3,4)$-torus knot have knot Floer homology of rank 5, as do certain cables of the trefoil).

\parpic[r]{\begin{tikzpicture}[scale=1,>=latex]
\node at (0.1,0) {$\bullet$}; 
\node at (-0.1,0) {$\bullet$}; 
\node at (-0.4,1) {$\bullet$}; \node at (-0.2,1) {$\bullet$}; \node at (0,1) {$\bullet$}; \node at (0.2,1) {$\bullet$}; \node at (0.4,1) {$\bullet$}; 
\node at (-0.1,2) {$\bullet$}; 
\node at (0.1,2) {$\bullet$}; 
\draw[thick,->,shorten <=3pt,shorten >=3pt] (-0.1,2) -- (-0.4,1);\draw[thick,->,shorten <=3pt,shorten >=3pt] (0.1,2) -- (0.4,1);
\draw[thick,->,shorten <=3pt,shorten >=3pt] (-0.2,1) -- (-0.1,0);\draw[thick,->,shorten <=3pt,shorten >=3pt] (0.2,1) -- (0.1,0);
\end{tikzpicture}}
By contrast, as a consequence of Theorem \ref{thm:ribbon}, there is an infinite family of distinct knots in $S^3$ sharing the knot Floer homology of the twist knot $6_1$. This results from the fact that this knot is ribbon; see Figure \ref{fig:bandsum}.  That is, if $K_0=6_1$ then by twisting the ribbon disk to obtain $\{K_i\}_{i\in\Z}$ we have that $\HFKa(K_i)$ is completely determined by the Alexander polynomial $-2t^{-1}+5-2t$ (and vanishing signature). In particular, this is an infinite family of distinct genus one knots with identical knot Floer homology of rank 9 (the subquotient complex $\Filt^{-1}(\{0\}\times \Z)$ is shown on the right).

There is an obvious question worth recording.
\begin{question}Are there only finitely many knots in $S^3$ with knot Floer homology of rank 5? Of rank 7?\end{question}

\subsection{Families of knots indistinguishable by homological invariants}\label{sub:Kanenobu} Kanenobu gave examples of distinct knots with identical HOMFLY polynomial \cite{Kanenobu1986}. We will denote these by $K_{p,q}$ where $p,q\in\Z$ record the number of full twists in two different locations on a ribbon disk; see Figure \ref{fig:kanenobu}. The knot $K_{0,0}$ is a connect sum of figure eight knots and, more generally, $K_{p,q}$ is a symmetric union of figure eight knots. See \cite{Kanenobu1986} for details; see \cite{Watson2007} for for various generalizations of the construction. From the foregoing discussion, it is worth recording that the $K_{p,q}$ are fibered and, apart from the connect sum for $(p,q)\ne(0,0)$, are all hyperbolic as well \cite[Theorem 2]{Kanenobu1986}. 

\begin{figure}[ht!] \centering
\labellist \large
\pinlabel $p$ at 111 110
\pinlabel $q$ at 252 110
	\endlabellist
\includegraphics[scale=0.4]{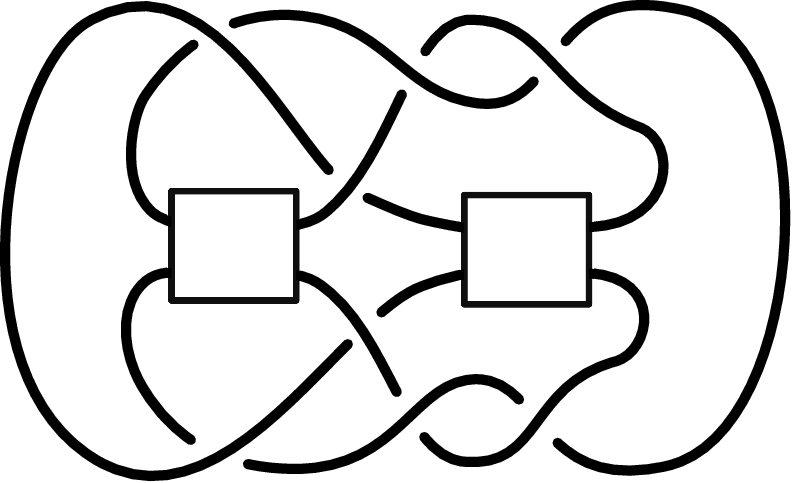}
\caption{The Kanenobu knot $K_{p,q}$, where $p,q\in\bZ$ record the number of full-twists on two strands following the convention in Figure \ref{fig:bandsum} (note that this convention agrees with \cite{Kanenobu1986} but differs from that of \cite{Watson2007}). Notice that $K_{p,q}$ may be realized as the band sum of two unknots in two different ways by cutting the band at either of the two twist sites. The infinite family of knots considered in Theorem \ref{thm:fail} arise for integers $p=-q=n$.}\label{fig:kanenobu}
\end{figure}

Let $K_n$ be the Kanenobu knot $K_{n,-n}$ for $n\in\Z$. This is an infinite family of distinct ribbon knots \cite[Lemma 2]{Kanenobu1986}. This final result is recorded for posterity.

\begin{theorem}\label{thm:fail} The following homological invariants fail to separate the knots $\{K_n\}_{n\in\Z}$:  
\begin{itemize}
\item[(i)] Khovanov homology
\item[(ii)] odd Khovanov homology
\item[(iii)] $sl(N)$ homology
\item[(iv)] HOMFLY homology
\item[(v)] knot Floer homology
\end{itemize}
\end{theorem}
\begin{proof}
The case (i) is established by the second author \cite{Watson2007}, and a modification of the argument is used with Greene in \cite{GW2011} to obtain (ii). Both (iii) and (iv) are results of Lobb \cite{Lobb2011}. Finally, since each $K_n$ is ribbon, (v) is an application of Theorem \ref{thm:ribbon}.
\end{proof}

\begin{footnotesize}
\subsection*{Acknowledgements} 
This paper began with a conversation in 2007; we apologize for how long it has taken to get it written. As the work now represents an effort to collect results potentially bound for folklore, the main theorems owe a debt to others.  Jake Rasmussen, in particular, pointed out an infinite family of $3$-strand pretzel knots with identical Floer homology; this provided direct inspiration for Theorem \ref{thm:infinite}. Allison Moore and Laura Starkston independently arrived at infinite families of knots with the same Floer homology using similar techniques \cite{MooreStarkston}.  For Theorem \ref{thm:identity}, Jeremy Van Horn-Morris pointed out to the authors that the only right-veering surface diffeomorphism which remains right-veering upon inversion is the identity. Jake Rasmussen also observed that the second highest group of a knot admitting L-space surgeries is non-trivial, as well as the utility of his inequality for the $h$-invariant, leading to the proof of Theorem \ref{thm:obstruction}.   We also thank many people whose interest and input  led to the existence of this note: Sebastian Baader, Josh Greene, Eli Grigsby, Jen Hom, Daniel Krasner, Allison Moore, Yi Ni, Brendan Owens and Zhongtao Wu, among others.  
\end{footnotesize}

\bibliographystyle{plain}
\bibliography{bibliography}

\end{document}